\documentclass[a4paper]{amsart}
\pdfoutput=1
\usepackage[all,hyperref]{paper_diening}
\renewcommand{\norm}[1]{\left\lVert#1\right\rVert}
\renewcommand{\abs}[1]{\left\lvert#1\right\rvert}

\usepackage[backend=biber,maxbibnames=5,sorting=nyt,maxalphanames=5,style=alphabetic,bibencoding=utf8,giveninits,url=false,isbn=false]{biblatex}
\AtBeginBibliography{\scriptsize}

\usepackage{amsmath,amsthm,amssymb}

\usepackage{epsfig}
\usepackage[utf8]{inputenc}
\usepackage[T1]{fontenc}
\usepackage{dsfont}
\usepackage{faktor}

\usepackage{tikz}
\usepackage{pgfplots}
\pgfplotsset{
  width=.65\linewidth,
  axis background/.style={fill=black!5!white},
  grid style={densely dotted,semithick},
  legend style={
    legend columns=1,
    legend pos=outer north east
  },
  compat=newest % compatibility for old pgfplots versions
}

\usepackage{mathrsfs}
\numberwithin{equation}{section}

\newcommand{\Div}{\divergence}

\newcommand{\dd}{\,\mathrm{d}}
\newcommand{\ds}{\dd s}
\newcommand{\dt}{\dd t}
\newcommand{\dx}{\dd x}

\providecommand{\seminormtmp}[2]{{#1[{#2}#1]}}
\providecommand{\seminorm}[1]{\seminormtmp{}{#1}}

\begin{document}

\title[Stochastic $p$-Stokes system]{Temporal regularity of symmetric stochastic $p$-Stokes systems}

\author{J\"{o}rn Wichmann}%
\address[J. Wichmann]{Department of Mathematics, University of Bielefeld,
  Postfach 10 01 31, 33501 Bielefeld, Germany}%
\email{jwichmann@math.uni-bielefeld.de}%
\thanks{ The research was funded by the Deutsche Forschungsgemeinschaft (DFG, German Research Foundation) – SFB 1283/2 2021 – 317210226. }

\begin{abstract}
 We study the symmetric stochastic $p$-Stokes system, $p \in (1,\infty)$, in a bounded domain. The results are two-folded.
 
First, we show that in the context of analytically weak solutions, the stochastic pressure -- related to non-divergence free stochastic forces -- enjoys almost $-1/2$ temporal derivatives on a Besov scale. 

Second, we verify that the velocity component~$u$ of strong solutions obeys $1/2$ temporal derivatives in an exponential Nikolskii space. Moreover, we prove that the non-linear symmetric gradient $V(\varepsilon u) = (\kappa + \abs{\varepsilon u})^{(p-2)/2} \varepsilon u$, $\kappa \geq 0$, has $1/2$ temporal derivatives in a Nikolskii space.

\end{abstract}

    % MSC

\maketitle

\tableofcontents
\section{Introduction}
Let $\mathcal{O} \subset \mathbb{R}^n$ be a bounded domain, $n \in \mathbb{N}$ and $T > 0$ be finite. Given an initial datum $u_0$ and a stochastic forcing~$G(\cdot) \dd W$, we are seeking for a velocity field $u: \Omega \times (0,T) \times \mathcal{O} \to \mathbb{R}^n$ and a pressure $\pi: \Omega \times (0,T) \times \mathcal{O} \to \mathbb{R}$ that satisfy the relations 
\begin{subequations} \label{ch:Stokes intro:p-Stokes-stoch}
\begin{alignat}{2} \label{ch:Stokes intro:p-Stokes-stoch01}
\dd u + \big( - \Div S(\varepsilon u)  + \nabla \pi \big)\dt  &= G(u) \dd W(t) \quad &&\text{ in } \Omega \times (0,T) \times \mathcal{O}, \\ \label{ch:Stokes intro:p-Stokes-stoch02}
\Div u &= 0 \quad &&\text{ in } \Omega \times (0,T) \times \mathcal{O}, \\\label{ch:Stokes intro:p-Stokes-stoch03}
u &= 0 \quad &&\text{ on } \Omega \times (0,T) \times \partial \mathcal{O}, \\ \label{ch:Stokes intro:p-Stokes-stoch04}
u(0) &= u_0 &&\text{ on } \Omega \times \mathcal{O}, \\\label{ch:Stokes intro:p-Stokes-stoch05}
\mean{\pi}_{\mathcal{O}} &= 0 &&\text{ in } \Omega \times (0,T),
\end{alignat}
\end{subequations}
where $S(\xi):= \left( \kappa + \abs{\xi} \right)^{p-2} \xi$, $\xi \in \mathbb{R}^{ n \times n}$, $p \in (1,\infty)$, $\kappa \geq 0$, $\varepsilon u = 2^{-1}\big( \nabla u + (\nabla u)^T \big)$ and $\mean{\pi}_{\mathcal{O}} = \dashint_{\mathcal{O}} \pi \dd x$. The stochastic input $G(\cdot) \dd W$ is rendered via a cylindrical Wiener process~$W$ on an abstract separable Hilbert space $U$ and a suitable noise coefficient $G$.

The system~\eqref{ch:Stokes intro:p-Stokes-stoch} is called stochastic symmetric $p$-Stokes system and describes the evolution of power-law fluids in the regime of laminar flow. Ultimately, one wants to substitute~\eqref{ch:Stokes intro:p-Stokes-stoch01} by
\begin{align}\label{ch:Stokes intro:p-Stokes-stochFull}
\dd u + \big( (u\cdot \nabla)u - \Div S(\varepsilon u)  + \nabla \pi \big)\dt  = G(u) \dd W(t)
\end{align}
to include convective effects. This might lead to turbulent flow. Maybe most prominent is the case $p=2$. Then the system reduces to the infamous stochastically forced Navier-Stokes equations. A derivation of this model and how meaningful choices of $G$ might look like can be found in e.g.~\cite{MR2050201} (see also~\cite{MR3607454} and the references therein). 

While the system~\eqref{ch:Stokes intro:p-Stokes-stoch} does not bother with the mathematical challenges of turbulence, it is still difficult to obtain improved regularity. There are several reasons:
\begin{enumerate}
\item Solutions to non-linear equations lack smoothness even for smooth data. 
\item Limited regularity of the stochastic data $G(\cdot) \dd W$ leads to limited regularity of the solution.
\item Stochastic integration in general Banach spaces causes technical restrictions.
\item Delicate interplay between the divergence free constrain~\eqref{ch:Stokes intro:p-Stokes-stoch02} and the pressure.
\end{enumerate}
In this article we focus on the temporal regularity for solutions to the system~\eqref{ch:Stokes intro:p-Stokes-stoch}. Before we comment on the available literature on regularity and finally state the main contributions of this paper, we briefly discuss some results related to the existence of non-Newtonian fluids. 

\subsection{Existence}
A mathematical theory for the existence of weak solutions to the deterministic counterpart of~\eqref{ch:Stokes intro:p-Stokes-stoch} enriched by the convective term was initiated by Ladyzhenskaya~\cite{MR0226907,MR0254401} and Lions~\cite{MR0259693}. Motivated by the works, many authors contributed to the construction of solutions, e.g. \cite{MR1203271,MNRR,MR1799487,MR2305828,DRW}. Ultimately leading to the construction of a weak solution for $p > 2n/(n+2)$. Below the compactness threshold $W^{1,p}_x \hookrightarrow \hookrightarrow L^2_x$ wild solutions may emerge, see e.g.~\cite{MR4073888}. 

The stochastic case is less known. An answer for solution independent noise coefficient was given in~\cite{MR2884052,MR2977990}. The authors require the condition $p > 9/5$ in the three dimensional setting. It was extended in~\cite{Breit2015} to cover more general noise coefficients and a unified condition $p > (2n+2)/(n+2)$. 

A construction for stochastic electro-rheological fluids, here $p$ also depends on $\Omega \times (0,T) \times \mathcal{O}$, is done in~\cite{MR4022286}.

\subsection{Regularity}
Regularity is needed in order to control discretizations. In particular, regularity moderates the speed of convergence of numerical schemes, see e.g.~\cite{MR4286257}. Therefore, and out of mathematical curiosity, many authors have investigated the regularity of solutions to~\eqref{ch:Stokes intro:p-Stokes-stoch} and related models. 

%In the following we will comment on the difficulties that arise through the non-linear structure of the system as well as the stochastic data. Moreover, we discuss related regularity results.

%Power-law fluids are discussed e.g. in~\cite{MR1603694,MR1754359,MR1810507,MR2004292,MR2122416,
%MR2824489,Br2,MR4188683}). Electro-rheological fluids, i.e., the power-law index $p$ also depends on time and space, are considered e.g. in~\cite{MR1810360,AcMi,diening2002theoretical,MR2077446,
%MR2346460,MR2317489,MR4022286}).

\subsubsection{Non-linear} In general, there is a barrier that does not allow for arbitrary high regularity even though the data is smooth, cf.~\cite{IwMa}. Still partial results can be obtained.

$C^{1,\alpha}$-regularity has been obtained in e.g.~\cite{MR1675954,AcMi,MR2346460,MR2317489}. H\"older regularity is sub-optimal from a numerical point of view. Instead, regularity of the expression $V(\varepsilon u) = (\kappa + \abs{\varepsilon u})^{(p-2)/2} \varepsilon u$ is more suitable for the investigation of numerical algorithms as observed by Barrett and Liu in~\cite{BaLi1}. Indeed, the operator $V$ measures the monotonicity of $S$ in a natural way, i.e., for all $A,B \in \mathbb{R}^{n\times n}$ it holds
\begin{align*}
\big( S(A) - S(B) \big) : (A- B) \eqsim \abs{V(A) - V(B)}^2.
\end{align*}

A first result for non-singular shear thinning fluids, $\kappa = 1$ and $p \in (1,2)$, has been obtained in~\cite{MR964511}. The author verifies local regularity $V(\varepsilon u) \in W^{1,2}_{\mathrm{loc}}$. 

A similar result has been obtained for shear thickening fluids, $p \geq 2$, in~\cite{Br2}. Additionally, local regularity for the non-linear tensor $S$ and the pressure~$\pi$ was found in~\cite{MR3208792}. They show $\nabla S(\varepsilon u), \nabla \pi \in L^{p'}_{\mathrm{loc}}$.

In~\cite{MR3275261} regularity transfer of the data to the solution in terms of Campanato and weighted BMO (bounded mean oscillation) spaces is shown.

Spatial regularity results for the parabolic $p$-Stokes system has been obtained in~\cite{AcMiSe}. The authors show that $\nabla u \in L^\infty_t L^2_{\mathrm{loc}}$ and $\nabla V(\varepsilon u) \in L^2_{\mathrm{loc}}(\mathcal{O}_T)$. It was extended to the $\phi$-Stokes problem in~\cite{MR3634347}.

As far as we know, improved temporal regularity for the parabolic symmetric $p$-Stokes system has not been analyzed in the literature. Only a few results that neglect the appearance of the pressure but keep the symmetric gradient are available. For example, the parabolic symmetric $p$-Laplace system has been studied by Frehse and Schwarzacher in~\cite{MR3411724} and by Burczak and Kaplick\'{y} in~\cite{MR3565947,}. In~\cite{MR3411724} it is shown that $\partial_t u \in B^{1/2}_{2,\infty} L^2_x$. Here $B$ denotes a Besov space (see Section~\ref{sec:Function spaces} for more details).

\subsubsection{Stochastic -- spatial regularity}
In general, the tools from deterministic theory are not available for stochastic equations. In particular, one cannot test the equation by a test function. Instead, one needs to perform an expansion for a suitable functional. This method is called It\^o's formula. While some test functions can be recovered in this way others cannot. Still, the deterministic regularity theory acts as a guiding example on what regularity can be expected.

It has been observed by Gess~\cite{Gess2012Strong} that stochastic partial differential equations, arising as a stochastically perturbed gradient flow of an energy, can lead to improved regularity results. He proposes conditions that ensure the existence of a strong solution.

Fortunately, the Helmholtz-Leray projection~$\Pi_{\Div}$ (see~\eqref{app:FuncSpace def:Helmholtz}) allows to reformulate~\eqref{ch:Stokes intro:p-Stokes-stoch01} and~\eqref{ch:Stokes intro:p-Stokes-stoch02} into 
\begin{subequations} \label{ch:intro intro:p-Stokes-stoch00New}
\begin{alignat}{2} \label{ch:intro  intro:p-Stokes-stoch01New}
\dd u  - \Pi_{\Div} \Div S(\varepsilon u) \dt  &= \Pi_{\Div} G(u) \dd W(t) \quad &&\text{ in } \Omega \times (0,T) \times \mathcal{O}, \\ \label{ch:intro  intro:p-Stokes-stoch02New}
\nabla \pi \dd t &= \Pi_{\Div}^\perp \big( \Div S(\varepsilon u) \dt + G(u) \dd W(t)  \big)\quad &&\text{ in } \Omega \times (0,T) \times \mathcal{O}.
\end{alignat}
\end{subequations}
In other words, the Helmholtz-Leray projection factorizes the equations for the velocity field and the pressure. It enables to seek for the velocity independently of the pressure.

Now it is possible to interpret~\eqref{ch:Stokes intro:p-Stokes-stoch01New} as a stochastically perturbed gradient flow of the energy
\begin{align} \label{ch:intro eq:Energy}
\mathcal{J}(u) := \int_\mathcal{O} \varphi_{\kappa} ( \abs{\varepsilon u}) \dd x
\end{align}
induced by the potential $\varphi_{\kappa}(t) := \int_0^t (\kappa + s)^{p-2} s \dd s$ on the space $W^{1,p}_{0,\Div}$. Indeed,~\eqref{ch:Stokes intro:p-Stokes-stoch01New} can be written as
\begin{align} \label{intro:GradFlow}
\dd u = - D \mathcal{J}(u) \dd t + \Pi_{\Div} G(u) \dd W(t) \quad \text{ in } \Omega \times (0,T) \times \mathcal{O},
\end{align}
where $D \mathcal{J}(u) = \Pi_{\Div} \Div S(\varepsilon u)$. Therefore, it is natural to look for regularity estimates that match the concept of strong solutions.

A similar approach is used by Breit and Gmeineder in~\cite{MR4022286}. They prove spatial regularity as in the deterministic case also for stochastically forced electro-rheological fluids and obtain $\nabla u \in L^2_\omega L^\infty_t L^2_x$ as well as $\nabla V(\varepsilon u) \in L^2_\omega L^2_t L^2_x$. An analogous line of argumentation can be used to establish spatial regularity for the $p$-Stokes system. 

\subsubsection{Stochastic -- temporal regularity}
In sharp contrast to spatial regularity, where similar results as in the deterministic theory can be expected, is temporal regularity for stochastic partial differential equations. In general, temporal regularity is substantially worse compared to deterministic equations. This is due to the action of the random data $G(\cdot) \dd W$ on the solution~$u$.

Even in the simplest case $G\equiv 1$, we expect $u$ to not exceed the regularity threshold induced by the Wiener process~$W$. A precise regularity result was derived by Hyt\"onen and Veraar in~\cite{VerHyt08}. They show that $\mathbb{P}$-a.s. $W \in B^{1/2}_{\Phi_2,\infty}$, where $\Phi_2(t) = \exp(t^2)-1$, and $W \not \in B^{1/2}_{p,q}$ for any $p \in [1,\infty]$ and $q \in [1,\infty)$. Therefore, a natural space for temporal regularity is the Nikolskii-space of functions with exponential second moments and $1/2$ derivatives.

Only recently, precise regularity results for stochastic integrals in $2$-smooth Banach spaces have been found by Ondreját and Veraar in~\cite{MR4116708}. Maybe most important, stochastic integration behaves as deterministic integration in the sense that improved regularity of the integrand leads to improved regularity of the stochastic integral itself. For example, if the integrand is bounded in time, then the stochastic integral is as regular as a Wiener process, i.e.,
\begin{align*}
G \in L^\infty \quad \Rightarrow \quad \int_0^{\cdot} G \dd W(t) \in B^{1/2}_{\Phi_2,\infty}.
\end{align*}

Exactly this fact is used in~\cite{wichmann2021temporal} to derive temporal regularity for the stochastic $p$-Laplace system. In particular, we manage to recover the limiting temporal regularity $u \in L^2_\omega B^{1/2}_{\Phi_2,\infty} L^2_x$. In other words, the regularity of $W$ fully transfers to regularity of $u$. Additionally, exploiting the $V$-coercivity, it is possible to derive non-linear gradient regularity $V(\nabla u) \in L^2_\omega B^{1/2}_{2,\infty} L^2_x$.

The stochastic $p$-Stokes system is even worse. A major obstacle in the derivation of higher regularity for the system~\eqref{ch:Stokes intro:p-Stokes-stoch} is the pressure. If $G(\cdot) \dd W$ is not divergence-free, then the pressure is very rough in time (see e.g.~\cite{MR2004282,MR3022227}). This can be seen by formally decomposing the pressure into a deterministic and a stochastic component $\pi = \pi_{\mathrm{det}} + \pi_{\mathrm{sto}}$, where
\begin{align*}
\nabla \pi_{\mathrm{det}} &= \Pi_{\mathrm{div}}^{\perp} \Div S(\varepsilon u), \\
\nabla \pi_{\mathrm{sto}} &= \Pi_{\mathrm{div}}^{\perp} G(u) \dot{W}.
\end{align*} 

While the deterministic pressure can be treated by classical arguments, we observe that the stochastic pressure looks like white noise in time. This leads to severe difficulties in numerical discretizations as soon as we approximate the velocity field by discretely divergence-free fields rather than exactly divergence-free fields. More details can be found in~\cite{MR4274685,MR4286261} and the references therein.

\subsection{Main results}
The main contributions of this paper are two-folded.

\subsubsection{Temporal regularity of stochastic pressure for weak solutions}
We apply the general theory of monotone stochastic partial differential equations developed by Liu and R\"ockner~\cite{LiRo}, see also the book \cite{MR3410409}, to the $p$-Stokes system. More recently, the theory has been refined to cover more general equations, see e.g.~\cite{2022arXiv220601107R,2022arXiv220600230A}. In this way, existence of weak solutions is almost immediate. The main novelty is a careful consideration of the regularity of the stochastic pressure.  

\begin{theorem}\label{ch:Stokes thm:EstimateWeakSolution}
Let $p \in (1,\infty)$ and $\mathcal{O}$ be a bounded $C^2$-domain. Additionally, let $q \in [1,\infty)$, $u_0 \in L^{q}_\omega L^2_{\Div}$ be $\mathcal{F}_0$-measurable and Assumption~\ref{ch:Stokes ass:NoiseCoeffweak} be satisfied. Then there exists a unique weak solution $(u,\pi)$ to~\eqref{ch:Stokes intro:p-Stokes-stoch} such that
\begin{subequations}
\begin{align} \label{ch:Stokes eq:EstimateWeakSolution}
\mathbb{E} \left[ \sup_{t\in I} \norm{u(t)}_{L^2_x}^{q} + \left( \int_I \mathcal{J}\big(u(t) \big) \dd t \right)^{q/2} \right] &\lesssim \mathbb{E} \left[ \norm{u_0}_{L^2_x}^{q} \right] + 1.
\end{align}
Moreover, $\pi = \pi_{\mathrm{det}} + \pi_{\mathrm{sto}} $ with
\begin{align}
\label{ch:Stokes eq:EstimateWeakSolution02}
\norm{\pi_{\mathrm{det}}}_{L^{p'q /2}_\omega L^{p'}_t L^{p'}_{x}}^{p'  /2} &\lesssim  \norm{u_0}_{L^q_\omega L^2_x}+ 1, \\ \label{ch:Stokes eq:EstimateWeakSolution03}
\norm{\pi_{\mathrm{sto}}}_{L^{q}_\omega B^{-1/2 - \varepsilon}_{s,r}  W^{1,2}_x} &\lesssim  \norm{u_0}_{L^q_\omega L^2_x} + 1, 
\end{align}
for any $s, r \in (1,\infty)$ and $\varepsilon >0$.
\end{subequations} 
\end{theorem}

\subsubsection{Temporal regularity of velocity field for strong solutions}
We exploit the gradient flow structure~\eqref{intro:GradFlow} on the space of solenoidal vector fields in order to use the existence theory of strong solutions initiated by Gess~\cite{Gess2012Strong}. In doing so, we provide a sufficient condition for the existence of strong solutions to the $p$-Stokes system. 

\begin{theorem} \label{ch:Stokes thm:StrongEnergy}
Let the assumptions of Theorem~\ref{ch:Stokes thm:EstimateWeakSolution} be satisfied. Additionally, let $q \in [1,\infty)$, $p \geq 2$, $ \mathcal{J}(u_0) \in L^{q}_\omega$ be $\mathcal{F}_0$-measurable and $G$ satisfy Assumption~\ref{ch:Stokes ass:NoiseCoeffstrong}. Then there exists a unique strong solution $(u,\pi)$ to~\eqref{ch:Stokes intro:p-Stokes-stoch}. Moreover, it holds
\begin{align} \label{ch:Stokes eq:StrongEnergy}
\mathbb{E} \left[ \left( \sup_{t \in I} \mathcal{J}\big( u(t) \big) + \int_0^T \int_{\mathcal{O}} \abs{\Pi_{\Div} \Div S(\varepsilon u)}^2 \dd x \dd t \right)^{q }\right] \lesssim \mathbb{E} \left[ \mathcal{J}\big( u_0 \big)^{q} \right] + 1.
\end{align}
\end{theorem}
 
We extend the techniques developed in~\cite{wichmann2021temporal} to obtain improved temporal regularity for the velocity field and its non-linear symmetric gradient. 

\begin{theorem} \label{ch:Stokes thm:BesovReg}
Let $p \in (1,\infty)$, Assumption~\ref{ch:Stokes ass:NoiseCoeffweak} be satisfied and $(u,\pi)$ be a strong solution to~\eqref{ch:Stokes intro:p-Stokes-stoch}. 
\begin{enumerate}
\item \label{it:a}(Polynomial integrability in time) Let $q \in [1,\infty)$. Then $u \in L^{2}_\omega B^{1/2}_{q,\infty} L^2_x$ and
\begin{align}\label{eq:mainNikolski01}
\norm{u}_{L^{2}_\omega B^{1/2}_{q,\infty} L^2_x} \lesssim  \norm{u}_{L^{2}_\omega L^\infty_t L^2_x} + \norm{\Pi_{\Div} \Div S(\varepsilon u)}_{L^2_\omega L^2_t L^2_x} + 1.
\end{align}
\item \label{it:b}(Exponential integrability in time) Let $u \in L^{N_2}_\omega L^\infty_t L^2_x$. Then $u \in L^{2}_\omega B^{1/2}_{\Phi_2,\infty} L^2_x$ and
\begin{align} \label{eq:mainNikolski}
\norm{u}_{L^{2}_\omega B^{1/2}_{\Phi_2,\infty} L^2_x} \lesssim  \norm{u}_{L^{N_2}_\omega L^\infty_t L^2_x} + \norm{\Pi_{\Div} \Div S(\varepsilon u)}_{L^2_\omega L^2_t L^2_x} + 1.
\end{align}
\end{enumerate}
\end{theorem}

\begin{remark}
Theorem~\ref{ch:Stokes thm:BesovReg} works for the full range $p \in (1,\infty)$. However, the existence of a strong solution is still open for $p \in (1,2)$.

Exponential integrability requires slightly stronger moment estimates. In a similar way to the a priori bound~\eqref{ch:Stokes eq:EstimateWeakSolution} for weak solutions, one can establish an estimate of the form
\begin{align*}
\norm{u}_{L^{N_2}_{\omega} L^\infty_t L^2_x} \lesssim \norm{u_0}_{L^{N_2}_{\omega} L^2_x} + 1.
\end{align*}
In other words, if the initial condition has $N_2$ moments, so does the solution at later times and the assumption of part~\ref{it:b} is satisfied. 
\end{remark}

\begin{theorem} \label{ch:Stokes thm:NikolskiiGrad}
Let $p \geq 2$, Assumption~\ref{ch:Stokes ass:NoiseCoeffstrong} be satisfied and $(u,\pi)$ be a strong solution to~\eqref{ch:Stokes intro:p-Stokes-stoch}. Additionally, assume $\mathcal{J}(u) \in L^1_\omega L^\infty_t$. Then $V(\varepsilon u) \in L^2_\omega B^{1/2}_{2,\infty} L^2_x$ and
\begin{align}
\norm{V(\varepsilon u)}_{L^2_\omega B^{1/2}_{2,\infty} L^2_x}^2 \lesssim  \mathbb{E} \left[ \int_{I} \norm{ \Pi_{\Div} \Div S(\varepsilon u)}_{L^2_x}^2 \dd t\right]+ \mathbb{E} \left[ \sup_{t\in I} \mathcal{J}\big(u(t)\big) \right]+ 1.
\end{align}
\end{theorem}

\subsection{Outline}
In Section~\ref{sec:math_setup} we introduce the mathematical framework. In particular, we discuss function spaces, the Helmholtz decomposition, the Bogovskii operator, potentials, mapping properties of stochastic integrals and gradient flows. 

Section~\ref{sec:RegularityWeak} deals with the construction of weak solutions. Moreover, regularity properties of the pressure are discussed.

Strong solutions are constructed in Section~\ref{sec:ExistenceStrong}.

In Section~\ref{sec:RegularityStrong} we present the temporal regularity of the velocity and the non-linear gradient. 

More details on Besov norms are presented in the Appendix~\ref{app:Besov}.

\section{Mathematical setup} \label{sec:math_setup}
Let $\mathcal{O} \subset \mathbb{R}^n$, $n \in \mathbb{N}$, be a bounded domain (further assumptions on $\mathcal{O}$ will be needed for the stability of the Helmholtz-Leray projection and the Bogovskii operator). For some given $T>0$ we denote by $I := (0,T)$ the time interval and write $\mathcal{O}_T := I \times \mathcal{O}$ for the time space cylinder. Moreover let $\left(\Omega,\mathcal{F}, (\mathcal{F}_t)_{t\in I}, \mathbb{P} \right)$ denote a stochastic basis, i.e., a probability space with a complete and right continuous filtration $(\mathcal{F}_t)_{t\in I}$. We write $f \lesssim g$ for two non-negative quantities $f$ and $g$ if $f$ is bounded by $g$ up to a multiplicative constant. Accordingly we define $\gtrsim$ and $\eqsim$. We denote by $c$ a generic constant which can change its value from line to line. Inner products and duality pairings are denoted by $\left( \cdot, \cdot \right)$ and $\langle \cdot, \cdot \rangle$, respectively.

\subsection{Function spaces} \label{sec:Function spaces}
As usual, for $1\leq q < \infty$ we denote by $L^q(\mathcal{O})$ the Lebesgue space and $W^{1,q}(\mathcal{O})$ the Sobolev space. Moreover, $W^{1,q}_0(\mathcal{O})$ denotes the Sobolev spaces with zero boundary values. It is the closure of $C^\infty_0(\mathcal{O})$ (smooth functions with compact support) in the $W^{1,q}(\mathcal{O})$-norm. We denote by $W^{-1,q'}(\mathcal{O})$ the dual of $W^{1,q}_0(\mathcal{O})$. The space of mean-value free Lebesgue functions is denoted by $L^q_0(\mathcal{O})$. The space of smooth, compactly supported and divergence-free vector fields is called $C^\infty_{0,\Div}(\mathcal{O})$ and its closure within the $W^{1,p}$-norm is abbreviated by $W^{1,p}_{0,\Div}(\mathcal{O})$. We do not distinguish in the notation between scalar-, vector- and matrix-valued functions.

For a Banach space $\left(X, \norm{\cdot}_X \right)$ let $L^q(I;X)$ be the Bochner space of Bochner-measurable functions $u: I \to X$ satisfying $t \mapsto \norm{u(t)}_X \in L^q(I)$. Moreover, $C^0(\overline{I};X)$ is the space of continuous functions with respect to the norm-topology. We also use $C^{0,\alpha}(\overline{I};X)$ for the space of $\alpha$-H\"{o}lder continuous functions. Given an Orlicz-function $\Phi: [0,\infty] \to [0,\infty]$, i.e. a convex function satisfying $ \lim_{t \to 0} \Phi(t)/t = 0$ and $\lim_{t \to \infty} \Phi(t)/t = \infty$ we define the Luxemburg-norm 
\begin{align*}
\norm{u}_{L^\Phi(I;X)} := \inf \left\{ \lambda > 0 : \int_I \Phi \left( \frac{\norm{u}_X}{\lambda} \right) \ds \leq 1 \right\}.
\end{align*}
The Orlicz space $L^\Phi(I;X)$ is the space of all Bochner-measurable functions with finite Luxemburg-norm. For more details on Orlicz-spaces we refer to \cite{DiHaHaRu}. Given $h \in I$ and $u :I \to X$ we define the difference operator $\tau_h: \set{u: I \to X} \to \set{u: I\cap I - \set{h} \to X} $ via $\tau_h(u) (s) := u(s+h) - u(s)$. The Besov-Orlicz space $B^\alpha_{\Phi,r}(I;X)$ with differentiability $\alpha \in (0,1)$, integrability $\Phi$ and fine index $r \in [1,\infty]$ is defined as the space of Bochner-measurable functions with finite Besov-Orlicz norm $\norm{\cdot}_{B^\alpha_{\Phi,r}(I;X)}$, where
\begin{align} \label{def:Besov01}
\begin{aligned}
\norm{u}_{B^\alpha_{\Phi,r}(I;X)} &:= \norm{u}_{L^{\Phi}(I ;X)} + \seminorm{u}_{B^{\alpha}_{\Phi,r}(I;X)}, \\
\seminorm{u}_{B^\alpha_{\Phi,r}(I;X)} &:= \begin{cases} \left( \int_{I} h^{-r\alpha} \norm{\tau_h u}_{L^\Phi(I\cap I - \set{h};X)}^r \frac{\dd h}{h} \right)^\frac{1}{r} & \text{ if } r \in [1,\infty), \\
\esssup_{h \in I} h^{-\alpha} \norm{\tau_h u}_{L^\Phi(I\cap I - \set{h};X)} & \text{ if } r = \infty.
\end{cases} 
\end{aligned}
\end{align} 
The case $r = \infty$ is commonly called Nikolskii-Orlicz space and abbreviated by $N^{\alpha,\Phi} = B^{\alpha}_{\Phi,\infty}$. When $\Phi(t) = t^p$ we write $B^\alpha_{p,r}(I;X)$ and call it Besov space. If $X$ is reflexive, we denote by $B^{-\alpha}_{p',q'}(I;X') = \big( B^{\alpha}_{p,q}(I;X) \big)'$.

Similarly, given a Banach space $\left(Y, \norm{\cdot}_Y \right)$, we define $L^q(\Omega;Y)$ as the Bochner space of Bochner-measurable functions $u: \Omega \to Y$ satisfying $\omega \mapsto \norm{u(\omega)}_Y \in L^q(\Omega)$. The space $L^q_{\mathcal{F}}(\Omega \times I;X)$ denotes the subspace of $X$-valued progressively measurable processes. We abbreviate the notation $L^q_\omega L^q_t L^q_x := L^q(\Omega;L^q(I;L^q(\mathcal{O}))) $ and $L^{q-} := \bigcap_{r< q} L^r$.

\subsection{Helmholtz-decomposition} \label{sec:Helmholtz}
The Helmholtz decomposition is a powerful tool in the analysis of fluids. It allows for a complete decoupling of the governing equations of the velocity field and the pressure. For a recent survey we refer to~\cite{bhatia2012helmholtz}. More details can be found e.g. in~\cite[Chapter~2 Section~3]{MR1855030}.

The general idea is to decompose a vector field into a divergence-free vector field and a gradient of a potential. Let $v \in L^2_x$ and $G_v$ be a solution to
\begin{align} \label{app:FuncSpace def:WeakLaplace}
\forall \xi \in W^{1,2}_x: \quad \left( \nabla G_v, \nabla \xi  \right) = \left(v, \nabla \xi \right).
\end{align}
Note that $G_v$ is defined up to a constant. We will prescribe the constant by the condition $\mean{G_v}_{\mathcal{O}} = 0.$

 Now, we define the Helmholtz-Leray projection~$\Pi_{\Div}$ and its orthogonal complement~$\Pi_{\Div}^\perp$ by
\begin{subequations}\label{app:FuncSpace def:Helmholtz}
\begin{align} \label{app:FuncSpace def:HelmholtzDiv}
\Pi_{\Div} v &:= v - \nabla G_v, \\ \label{app:FuncSpace def:HelmholtzGrad}
\Pi_{\Div}^\perp v&:= \nabla G_v.
\end{align}
\end{subequations}

In order to understand the range of the operator $\Pi_{\Div}$ and $\Pi_{\Div}^\perp$ we introduce the spaces
\begin{subequations}
\begin{align}
L^2_{\Div} &:= \set{v \in L^2_x | \, \Div v = 0 \text{ in } \mathcal{O}\, \wedge v\cdot \eta = 0 \text{ on } \partial \mathcal{O}}, \\
\big( L^2_{\Div} \big)^\perp &:= \set{v \in L^2_x | \, \exists G \text{ such that }  v = \nabla G \text{ in } \mathcal{O}}.
\end{align}
\end{subequations}

The next lemma can be found e.g. in~\cite[Theorem~1.4]{MR603444}. 
\begin{lemma}[Helmholtz decomposition] \label{app:FuncSpace lem:Helmholtz}
$L^2_x$ splits into the direct sum of the closed subspaces $L^2_{\Div}$ and $(L^2_{\Div})^\perp$, i.e.,
\begin{align}
L^2_x = L^2_{\Div} \oplus (L^2_{\Div})^{\perp}.
\end{align}
Moreover, $\Pi_{\Div}$ and $\Pi_{\Div}^\perp$ are the $L^2$-orthogonal projections onto $L^2_{\Div}$ and $(L^2_{\Div})^\perp$, respectively.
\end{lemma}

\begin{remark}
Since $\Pi_{\Div}$ and $\Pi_{\Div}^\perp$ are $L^2$-orthogonal projections, an equivalent definition to~\eqref{app:FuncSpace def:Helmholtz} is given by
\begin{align}
\Pi_{\Div} v = \argmin_{w \in L^2_{\Div}} \norm{w-v}_{L^2_x} \quad \text{ and } \quad \Pi_{\Div}^\perp v = \argmin_{w \in (L^2_{\Div})^\perp} \norm{w-v}_{L^2_x}.
\end{align}
\end{remark}

Stability of the Helmholtz-decomposition is particularly important, when it comes to the reconstruction of the pressure. Clearly, as an $L^2$-projection it trivially holds
\begin{align*}
\norm{\Pi_{\Div} v}_{L^2_x} \leq \norm{v}_{L^2_x} \quad \text{ and } \quad \bignorm{\Pi_{\Div}^\perp v}_{L^2_x} \leq \norm{v}_{L^2_x}.
\end{align*}

Higher order stability is closely linked to regularity of solutions to the Laplace equation. Indeed, if $\Div v \in L^2_x$, then~\eqref{app:FuncSpace def:WeakLaplace} is equivalent to the Neumann-Laplacian
\begin{align} \label{app:FuncSpace def:WeakLaplace02}
\begin{cases}
\Delta G_v = \Div v \quad &\text{ in } \mathcal{O},\\
\partial_{\eta} G_v = v \cdot \eta \quad &\text{ on } \partial \mathcal{O},
\end{cases}
\end{align}
i.e., $G_v =  \Delta^{-1} \Div v $ is a weak solution to~\eqref{app:FuncSpace def:WeakLaplace02}.

Therefore, it is not surprising that improved stability properties of the Laplacian transfer to stability of the Helmholtz-Leray projection, cf. \cite[Remark~1.6]{MR603444}.

\begin{theorem} \label{app:FuncSpace thm:GradientStabilityHelm}
Let $\mathcal{O}$ be a $C^{k+1}$-domain and $q\in (1,\infty)$. Then
\begin{align} \label{app:FuncSpace eq:GradientStabilityHelm}
\norm{\Pi_{\Div} v}_{W^{k,q}_x} \lesssim \norm{v}_{W^{k,q}_x} \quad \text{ and } \quad \norm{\Pi_{\Div}^\perp v}_{W^{k,q}_x} \lesssim \norm{\Div v}_{W^{k-1,q}_x}.
\end{align}
\end{theorem}
Theorem~\ref{app:FuncSpace thm:GradientStabilityHelm} can be proven using the fundamental solution of the Neumann-problem~\eqref{app:FuncSpace def:WeakLaplace02} and Calderon-Zygmund estimates. For example the result for Bessel-potential spaces is presented in~\cite[Lemma~3.6 \& Lemma~3.8]{MR1950829}.

\begin{corollary} \label{cor:Dual}
Let $\mathcal{O}$ be a bounded $C^2$-domain and $q \in (1,\infty)$. Then 
\begin{enumerate}
\item $\nabla \Pi_{\Div}: W^{1,q}_{0,\Div} \to L^q_x$ is bounded,
\item $\Pi_{\Div} \Div : L^{q'}_x \to \big( W^{1,q}_{0,\Div} \big)'$ is bounded,
\item for all $u \in W^{1,q}_{0,\Div}$ and $S \in L^{q'}_x$ 
\begin{align} \label{eq:dual}
\langle -\Pi_{\Div} \Div S, u \rangle =\int_{\mathcal{O}} S :\nabla \Pi_{\Div} u \dd x.
\end{align}
\end{enumerate}
\end{corollary}
\begin{proof}
The claims immediately follow by Theorem~\ref{app:FuncSpace thm:GradientStabilityHelm} and the density of smooth functions.
\end{proof}

\subsection{Bogovskii operator}
The right inverse of the divergence is nowadays called Bogovskii operator. It traces back to his work~\cite{MR631691} and is a useful tool for the regularity analysis of the pressure.

\begin{theorem}[{Bogovskii's operator \cite[Theorem~2.5]{MR2240056}}] \label{app:Aux thm:Bogovskii}
Let $q \in (1,\infty) $ and $\mathcal{O} \subset \mathbb{R}^n$ be a bounded domain with a locally Lipschitz boundary. Then there exists $\mathcal{B}: C^\infty_{0,x} \to C^\infty_{0,x}$ such that 
\begin{align}
\Div \mathcal{B} g = g, \quad g \in L^q_x, \quad \mean{g}_{\mathcal{O}} = 0.
\end{align} 
Moreover, $\mathcal{B}$ extends continuously to a bounded operator from $W^{s,q}_{0,x}$ to $W^{s+1,q}_{0,x}$ provided $s > -2 + 1/q$.
\end{theorem}

\subsection{Monotonicity and potentials}
A continuous, convex and strictly increasing function
$\phi\,:\, [0,\infty) \to [0,\infty)$ satisfying
\begin{align*}
  \lim_{t\rightarrow0}\frac{\phi(t)}{t}=
  \lim_{t\rightarrow\infty}\frac{t}{\phi(t)}=0
\end{align*}
is called an $N$-function.

We say that $\phi$ satisfies the $\Delta_2$--condition, if there exists $c > 0$ such that for all $t \geq 0$ holds $\phi(2t) \leq c\, \phi(t)$. By $\Delta_2(\phi)$ we denote the smallest such constant. Since $\phi(t) \leq \phi(2t)$ the $\Delta_2$-condition is equivalent to $\phi(2t) \eqsim \phi(t)$ uniformly in $t$. Note that if $\Delta_2(\phi) < \infty$ then $\phi(t) \eqsim \phi(c\,t)$ uniformly in $t\geq 0$ for any fixed $c>0$. For a family $\phi_\lambda$ of $N$-functions we define
$\Delta_2(\set{\phi_\lambda}) := \sup_\lambda \Delta_2(\phi_\lambda)$.

By $\phi^*$ we denote the conjugate N-function of $\phi$, which is
given by $\phi^*(t) = \sup_{s \geq 0} (st - \phi(s))$. Then $\phi^{**}
= \phi$.

More details can be found e.g. in~\cite{DieE08,DiHaHaRu,DR,BelDieKre2012,DieForTomWan20}.

\begin{definition}
  \label{app:FuncSpace ass:phipp}
  Let $\phi$ be an N-function. 
  We say that $\phi$ is \emph{uniformly convex}, if
  $\phi$ is $C^1$ on $[0,\infty)$ and $C^2$ on $(0,\infty)$ and
  assume that 
  \begin{align}
    \label{app:FuncSpace eq:phipp}
    \phi'(t) &\eqsim t\,\phi''(t)
  \end{align}
  uniformly in $t > 0$. The constants hidden in $\eqsim$ are called the
  \emph{characteristics of~$\phi$}.
\end{definition}
Note that~\eqref{app:FuncSpace eq:phipp} is stronger than
$\Delta_2(\phi,\phi^*)<\infty$. In fact, the $\Delta_2$-constants can
be estimated in terms of the characteristics of~$\phi$.

Associated to an uniformly convex $N$-function $\phi$ we define the tensors
\begin{align} \label{app:FuncSpace def:SandV}
   S( \xi)&:=\frac{\phi'(\abs{ \xi})}{\abs{ \xi}} \xi \quad \text{ and } \quad 
   V( \xi):=\sqrt{\frac{\phi'(\abs{ \xi})}{\abs{ \xi}}}\xi  \in \mathbb{R}^{n \times n}.
\end{align}

Sometimes it is convenient to change the growth of an $N$-function near $0$. One possibility is to introduce the shifted $N$-function $\phi_\kappa$ for $\kappa\geq 0$ by
\begin{align}
  \label{app:FuncSpace eq:def_shift}
  \varphi_{\kappa}(t) &:= \int_0^t \frac{\phi'(\kappa+s)}{\kappa+s} s \ds.
\end{align}
We define $S_\kappa$ and $V_\kappa$ analogously to~\eqref{app:FuncSpace def:SandV}. However, we neglect in the notation the dependence of $S_\kappa$ and $V_\kappa$ on $\kappa$, since the characteristic of $\varphi_{\kappa}$ is uniform in~$\kappa \geq 0$.
\begin{lemma}[{\cite[Lemma~27]{DieE08}}]
Let $\phi$ be an N-function with $\Delta_2(\phi, \phi^*) < \infty$. Then the family $\set{\varphi_\kappa}_{\kappa \geq 0}$ satisfies $\sup_{\kappa \geq 0} \Delta_2\big( \varphi_\kappa, (\varphi_\kappa)^* \big) < \infty$.
\end{lemma}
\begin{lemma}[{Equivalence lemma \cite[Lemma~3]{DieE08}}]
  \label{app:FuncSpace lem:hammer}
  We have
  \begin{align*}
    \begin{aligned} 
      \big({S}(P) - {S}(Q)\big) : \big(P-Q\big) &\eqsim \bigabs{V(P)-V(Q)}^2\\
      &\eqsim \phi_{\abs{P}}(\abs{P-Q}) \\
 &\eqsim \phi''\big(\abs{P}+\abs{Q} \big)\abs{ P -  Q}^2
    \end{aligned}
  \end{align*}
  uniformly in $ P,  Q \in \setR^{n \times n}$.  Moreover,
  uniformly in $ Q \in \setR^{n \times n}$,
  \begin{align*}
     S(Q) : Q &= \abs{V(Q)}^2\eqsim \phi(\abs{ Q})\\
    \abs{{S}(P) - {S}(Q)}&\eqsim \big(\phi_{\abs{P}}\big)'(\abs{P-Q}).
  \end{align*}
  The constants depend only on the characteristics of $\phi$.
\end{lemma}

\begin{lemma}[{\cite[Lemma~32]{DieE08}}] \label{app:FuncSpace lem:GenYoung}
Let $\varphi$ be an uniformly convex $N$-function. Then for all $\delta >0$ there exists $C_\delta > 0$ such that for all $t,\,u \geq 0$
\begin{subequations}
\begin{align}
    tu &\leq \delta \varphi( t ) + C_\delta \varphi^*( u) \\
    t \varphi'(u) +  u\varphi'(t) &\leq \delta \varphi( t ) + C_\delta \varphi( u ).
\end{align}
\end{subequations}
\end{lemma}
The following inequality follows immediately from Lemma~\ref{app:FuncSpace lem:hammer} and Lemma~\ref{app:FuncSpace lem:GenYoung}.
\begin{lemma}[Young type inequality]
  \label{app:FuncSpace lem:young}
  Let $\phi$ be an uniformly convex N-function. Then for each $\delta>0$ there exists $C_\delta \geq 1$ (only depending on~$\delta$ and the characteristics of~$\phi$) such that
  \begin{align*}
    \big({ S}( P) - { S}( Q)\big) :
    \big( R- Q \big) &\leq \delta \bigabs{V(P)-V(Q)}^2+C_\delta \bigabs{V(R)-V(Q)}^2
  \end{align*}
  for all $P, Q, R\in\mathbb{R}^{N\times n}$.
\end{lemma}

\begin{lemma}[{Change of Shift  \cite[Lemma~42]{DieForTomWan20}}]
 \label{app:FuncSpace lem:shift_ch}
 Let $\phi$ be an uniformly convex N-function. Then for each $\delta>0$ there
 exists $C_\delta \geq 1$ (only depending on~$\delta$ and the
 characteristics of~$\phi$) such that
  \begin{align*}
    \phi_{\abs{ P}}(t) &\leq C_\delta\, \phi_{\abs{ Q}}(t)
    +\delta\, \abs{ V( P) -  V( Q)}^2,
  \end{align*}
 for all $ P, Q\in\setR^{n \times n}$ and $t\geq0$.
\end{lemma}

\subsection{Structural assumptions on the stochastic data} \label{ch:Stokes sec:Ass}
Let $U$ be some separable Hilbert space and $\set{u_j}_{j\in \mathbb{N}}$ be a complete orthonormal system of $U$. We render the stochastic input via a cylindrical noise on the abstract space $U$.
\begin{assumption}[Cylindrical Wiener process]\label{ch:Reg ass:CylindricalNoise}
We assume that $W$ is an $U$-valued cylindrical Wiener process with respect to $(\mathcal{F}_t)$. Formally $W$ can be represented as
\begin{align} \label{ch:Reg rep:W}
W = \sum_{j \in \mathbb{N}} u_j \beta^j,
\end{align}
where $\set{\beta^j}_{j\in \mathbb{N}}$ are independent $1$-dimensional standard Brownian motions.
\end{assumption}

We construct the noise coefficient~$G$ as a two parameter map. In its first parameter it acts as a Nemytskii operator whereas the dependence on the Wiener process is linear. 
\begin{definition}[Noise coefficient]
Let $\set{g_j}_{j\in \mathbb{N}} : \mathcal{O} \times \mathbb{R}^N \to \mathbb{R}^N$. We define $G$ by
\begin{align} \label{ch:Reg def:NoiseCoeff}
\begin{aligned}
G: L^2_{\mathcal{F}}(\Omega \times I; L^2_x) \times U &\to L^2_{\mathcal{F}}(\Omega \times I; L^2_x) , \\
(v,u) &\mapsto G(v)u := \sum_{j \in \mathbb{N}} g_j\big(\cdot,v(\cdot)\big)( u_j, u)_U.
\end{aligned}
\end{align}
\end{definition}

$G$ is uniquely determined by the sequence of functions $\set{g_j}$. In particular, regularity assumptions and summation properties need to be imposed on $\set{g_j}$ to obtain an operator $G$ that is stable on specific function spaces. 

The construction of the stochastic integral in the framework of Hilbert spaces can be done by the It\^o isometry and requires the integrand to be an Hilbert-Schmidt operator, i.e., if $G \in L^2_{\mathcal{F}}\big(\Omega \times I; L_2(U; L^2_x) \big)$, then 
\begin{align} \label{eq:StochIntegral}
\mathcal{I}_W(G) := \int_0^{\cdot} G \dd W(t) := \sum_{j\in \mathbb{N}} \int_0^{\cdot} G(u_j) \dd \beta^j(t)
\end{align}
converges in $L^2\big(\Omega; C(I; L^2_x) \big)$. Moreover, we have the equivalence
\begin{align} \label{eq:StochIntegralIto}
\mathbb{E} \left[ \sup_{t \in I} \norm{\mathcal{I}_W(G)}_{L^2_x}^2 \right] \eqsim \mathbb{E} \left[ \int_I \norm{G}_{L_2(U;L^2_x)}^2  \dd t\right].
\end{align}

The construction of the stochastic integral in general Banach spaces is delicate and one needs to look for a suitable generalization of the It\^o isometry. It turns out that $\gamma$-radonifying operators are natural in the context of stochastic integration. The following result was obtained by Ondreját and Veraar and contains optimal stability of the stochastic integral driven by a cylindrical Wiener process for $\gamma$-radonifying operators with values in a separable $2$-smooth Banach space. For more details on $\gamma$-radonifying operators we refer to the survey~\cite{MR2655391}.

\begin{theorem}[{Stability of stochastic integrals \cite[Theorem~3.2]{MR4116708}}] \label{app:StochInt thm:Stability}
Let $E$ be a separable $2$-smooth Banach space. Additionally, let $q \in[1,\infty)$, $p \in (2,\infty]$, $\alpha = \tfrac{1}{2} - \tfrac{1}{p}$, $\Phi_2(t) := e^{t^2}-1$, $N_q(t):= t^q \ln^{q/2}(t+1)$ and $G \in L^0_\mathcal{F}\big(\Omega;L^p(I;\gamma(U;E))\big)$. Then
\begin{enumerate}
\item $\mathcal{I}_W(G) \in B^{\alpha}_{\Phi_2,\infty}(I;E)$ a.s.,
\item \label{app:StochInt it:2}$\left( \mathbb{E} \left[ \norm{\mathcal{I}_W(G)}_{B^{\alpha}_{q,\infty}(I;E)}^{2q} \right] \right)^\frac{1}{2q}\leq C \sqrt{q} \norm{G}_{L^{2q}(\Omega;L^{p}(I;\gamma(U;E)))}$,
\item $\left( \mathbb{E} \left[ \norm{\mathcal{I}_W(G)}_{B^{\alpha}_{\Phi_2,\infty}(I;E)}^{q} \right] \right)^\frac{1}{q}\leq C \sqrt{q} \norm{G}_{L^{\infty}(\Omega;L^{p}(I;\gamma(U;E)))}$,
\item $\norm{\mathcal{I}_W(G)}_{L^{\Phi_2}(\Omega;B^{\alpha}_{\Phi_2,\infty}(I;E)} \leq C \norm{G}_{L^\infty(\Omega;L^p(I;\gamma(U;E)))}$,
\item \label{app:StochInt it:5}$\norm{\mathcal{I}_W(G)}_{L^q(\Omega;B^{\alpha}_{\Phi_2,\infty}(I;E))} \leq C \sqrt{q} \norm{G}_{L^{N_q}(\Omega;L^p(I;\gamma(U;E)))}$.
\end{enumerate}
\end{theorem}

We are particularly interested in the cases $E = W^{1,p}_{0,x}$ for $p \geq 2$. This space is $2$-smooth, cf. \cite{MR3617459}. Naturally, if $E$ is $2$-smooth, then Hilbert-Schmidt operators are $\gamma$-radonifying, i.e., $L_2(U;E) \subset \gamma(U;E)$. Indeed,
\begin{align} \label{app:StochInt eq:RelGammaL2}
\norm{G}_{\gamma(U;E)}^2 := \mathbb{E}_{\gamma} \left[ \Biggnorm{\sum_{j\in \mathbb{N}} G(u_j)\gamma_j }_E^2 \right] \leq C_2^2 \sum_{j\in \mathbb{N}} \norm{G(u_j)}_E^2 = :C_2^2 \norm{G}_{L_2(U;E)}^2,
\end{align}
where $C_2$ is the martingale type $2$ constant of $E$ and $\set{\gamma_j}_{j\in \mathbb{N}} \sim \mathcal{N}(0,1)$ are independent and identically distributed. 

From now on we assume that $W$ is given by~\eqref{ch:Reg rep:W} and $G$ is of the form~\eqref{ch:Reg def:NoiseCoeff}.

\subsection{Gradient flow} \label{ch:Stokes sec:Grad}
The Helmholtz-Leray projection~$\Pi_{\Div}$ allows to reformulate~\eqref{ch:Stokes intro:p-Stokes-stoch01} and~\eqref{ch:Stokes intro:p-Stokes-stoch02} into 
\begin{subequations} \label{ch:Stokes intro:p-Stokes-stoch00New}
\begin{alignat}{2} \label{ch:Stokes intro:p-Stokes-stoch01New}
\dd u  - \Pi_{\Div} \Div S(\varepsilon u) \dt  &= \Pi_{\Div} G(u) \dd W(t) \quad &&\text{ in } (0,T) \times \mathcal{O}, \\ \label{ch:Stokes intro:p-Stokes-stoch02New}
\nabla \pi \dd t &= \Pi_{\Div}^\perp \left( \Div S(\varepsilon u) \dt + G(u) \dd W(t)  \right)\quad &&\text{ in } (0,T) \times \mathcal{O}.
\end{alignat}
\end{subequations}

Now it is possible to interpret~\eqref{ch:Stokes intro:p-Stokes-stoch01New} as a stochastically perturbed gradient flow of the energy
\begin{align} \label{ch:Stokes eq:Energy}
\mathcal{J}(u) := \int_\mathcal{O} \varphi_{\kappa} ( \abs{\varepsilon u}) \dd x
\end{align}
induced by the potential $\varphi_{\kappa}(t) := \int_0^t (\kappa + s)^{p-2} s \dd s$ on the space $W^{1,p}_{0,\Div}$.

It has been observed by Diening and Kreuzer~\cite{DieKre08} that the energy gap of $u \in W^{1,p}_x$ towards a minimizer~$ u^* \in \mathrm{arg min}_{v \in W^{1,p}_x} \mathcal{J}(v)$ is proportional to the distance measured in the $V$-quasi norm, i.e.,
\begin{align} \label{eq:energyGapOpt}
\mathcal{J}(u) - \mathcal{J}(u^*) \eqsim  \norm{V(\varepsilon u) - V(\varepsilon u^*)}_{L^2_x}^2.
\end{align}
In fact, the result can be generalized to energy gaps between arbitrary functions $u,v \in W^{1,p}_x$.

\begin{lemma} \label{lem:Taylor}
Let $u,v \in W^{1,p}_x$. Then 
\begin{align} \label{eq:StronglyConvex}
&\mathcal{J}(u) - \mathcal{J}(v) - D \mathcal{J}(v)[u-v] \eqsim \norm{V(\varepsilon u) - V(\varepsilon v)}_{L^2_x}^2.
\end{align}
\end{lemma}
\begin{proof}
Due to a Taylor expansion it holds
\begin{align*}
&\mathcal{J}(u) - \mathcal{J}(v) - D \mathcal{J}(v)[u-v] =\int_0^1 (1-\theta) D^2 \mathcal{J}\big(v + \theta (u-v) \big)[u-v, u-v]\dd \theta,
\end{align*}
where the first and second order Gateaux derivatives are given by
\begin{align} \label{eq:1stGateaux}
D \mathcal{J}(u)[v] &= \int_{\mathcal{O}} \frac{\varphi_{\kappa}'(\abs{\varepsilon u})}{\abs{\varepsilon u}} \varepsilon u : \varepsilon v \dd x
\end{align}
and
\begin{align} \label{eq:2ndGateaux}
\begin{aligned}
D^2 \mathcal{J}(u)[v,w] &= \int_{\mathcal{O}} \left( \varphi_{\kappa}''\big(\abs{\varepsilon u}) - \frac{\varphi_{\kappa}'\big( \abs{\varepsilon u} \big)}{\abs{\varepsilon u}} \right) \frac{\varepsilon u}{\abs{\varepsilon u}}: \varepsilon v \frac{\varepsilon u}{\abs{\varepsilon u}}: \varepsilon w \dd x \\
&\quad + \int_{\mathcal{O}}  \frac{\varphi_{\kappa}'\big( \abs{\varepsilon u} \big) }{\abs{\varepsilon u}}  \varepsilon v : \varepsilon w \dd x,
\end{aligned}
\end{align}
respectively. 

Following the proof of~\cite[Lemma~16]{DieKre08} we find
\begin{align} \label{eq:2ndProportion}
\int_0^1 (1-\theta) D^2 \mathcal{J}\big(v + \theta (u-v) \big)[u-v, u-v]\dd \theta \eqsim \norm{V(\varepsilon u) - V(\varepsilon v)}_{L^2_x}^2.
\end{align}
This establishes the claim.
\end{proof}
Clearly, if we insert a minimizer $v = v^*$ in~\eqref{eq:StronglyConvex}, then we recover~\eqref{eq:energyGapOpt}. However, the estimate~\eqref{eq:StronglyConvex} ensures that $\mathcal{J}$ is strongly convex.

In order to identify~\eqref{ch:Stokes intro:p-Stokes-stoch01New} as a perturbed gradient flow, we need to find a pointwise representation of the gradient $D \mathcal{J}(u)$.

\begin{proposition}\label{ch:Stokes prop:GradientFlow}
Let $\mathcal{O}$ be a bounded $C^2$-domain and $u \in W^{1,p}_{0,x}$. Then for all $\xi \in W^{1,p}_{0,\Div}$ 
\begin{align}\label{ch:Stokes eq:GradientFlowRough}
-D \mathcal{J}(u)[\xi] = \langle \Pi_{\Div} \Div S(\varepsilon u), \xi \rangle.
\end{align}

If additionally $\xi \in C^\infty_{0,\Div}$ and $\Pi_{\Div} \Div S(\varepsilon u) \in L^1_{\mathrm{loc}}$, then
\begin{align}\label{ch:Stokes eq:GradientFlow}
-D \mathcal{J}(u)[\xi]  = \int_{\mathcal{O}} \Pi_{\Div} \Div S(\varepsilon u) \cdot \xi  \dd x.
\end{align}
\end{proposition}
\begin{proof}
Let $\xi \in W^{1,p}_{0,\Div}$. Due to~\eqref{eq:1stGateaux}, ~\eqref{app:FuncSpace def:SandV} and the symmetry of $S(\varepsilon u)$ 
\begin{align*}
D \mathcal{J}(u)[\xi] &= \int_{\mathcal{O}} S(\varepsilon u) : \varepsilon \xi \dd x =\int_{\mathcal{O}}  S(\varepsilon u) : \nabla \xi \dd x .
\end{align*}
Next, we use that $\Pi_{\Div}$ is a projection, $\xi \in \mathcal{R}(\Pi_{\Div})$ and~\eqref{eq:dual},
\begin{align*}
\int_{\mathcal{O}}  S(\varepsilon u) : \nabla \xi \dd x = \int_{\mathcal{O}}  S(\varepsilon u) : \nabla \Pi_{\Div} \xi \dd x = \langle - \Pi_{\Div} \Div S(\varepsilon u) ,\xi \rangle.
\end{align*}
This establishes~\eqref{ch:Stokes eq:GradientFlowRough}.

The second part follows by Riesz's representation theorem.
\end{proof}

\begin{remark}
We want to emphasize that in general it is not possible to decouple the projection and the divergence, i.e., $\Pi_{\Div} \Div$ should be understood as one operator (defined by the duality relation~\eqref{eq:dual}) rather than the composition of two.

However, for smooth functions it is exactly the composition of the divergence and the Helmholtz-Leray projection.  
\end{remark}

\section{Temporal regularity of stochastic pressure for weak solutions} \label{sec:RegularityWeak}
The aim of this section is the discuss weak solutions to~\eqref{ch:Stokes intro:p-Stokes-stoch}. In particular, we define the concept of weak solutions and present a sufficient condition on the stochastic data for the existence of a unique weak solution~$(u,\pi)$. Moreover, we split the pressure into two terms. One is related to the diffusion operator~$S(\varepsilon u)$, while the second corresponds to the stochastic data.

\subsection{Weak solutions}
The concept of weak solutions is defined as follows.
\begin{definition} \label{def:weak}
Let $u_0 \in L^2_\omega L^2_{\Div}$ be $\mathcal{F}_0$-measurable. The tupel $(u,\pi)$ is called weak solution to \eqref{ch:Stokes intro:p-Stokes-stoch} if 
\begin{enumerate}
\item $u \in L^2_\omega C_t^0 L^2_{\Div} \cap L^p_\omega L^p_t W^{1,p}_{0,x}$ is $(\mathcal{F}_t)$-adapted,
\item  for all $t \in I$, $\xi \in C^\infty_{0,\Div}$ and $\mathbb{P}$-a.s. it holds
\begin{align} \label{ch:Stokes eq:p-Stokes_weak}
\int_{\mathcal{O}} (u(t) - u_0) \cdot \xi \dx + \int_0^t \int_{\mathcal{O}} S( \varepsilon u) :\nabla \xi \dx \ds = \sum_{j\in \mathbb{N}}  \int_0^t \int_{\mathcal{O}} g_j(\cdot,u)\cdot \xi \dx \dd \beta^j(s),
\end{align}
\item $\pi= \pi_{\mathrm{det}} + \pi_{\mathrm{sto}}$ with $\pi_{\mathrm{det}} \in L^{p'}_\omega L^{p'}_t L^{p'}_{0,x}$, $\pi_{\mathrm{sto}} \in L^{2}_\omega B^{-1/2}_{\Phi_2,\infty} \big( W^{1,2}_x \cap L^2_{0,x}  \big)$
\item and for all $\xi_{t,x} \in C^\infty_0(\mathcal{O}_T)$ and $\mathbb{P}$-a.s.
\begin{align}\label{ch:Stokes eq:p-Stokes_weakPressure}
\langle \pi, \Div \xi_{t,x} \rangle  = \int_{\mathcal{O}_T}  S(\varepsilon u): \nabla \Pi_{\Div}^\perp \xi_{t,x} \dd x \dd t +\int_{\mathcal{O}_T} \mathcal{I}_W\big(G(u)\big)  \partial_t \Pi_{\Div}^\perp \xi_{t,x} \dd x \dd t.
\end{align}
\end{enumerate}
\end{definition}

\subsection{A sufficient condition for weak solutions}
\begin{assumption}\label{ch:Stokes ass:NoiseCoeffweak}
We assume that $\set{g_j}_{j\in \mathbb{N}} \in C^0(\mathcal{O} \times \mathbb{R}^n; \mathbb{R}^n)$ with
\begin{enumerate}
\item (sublinear growth) for all $v \in L^2_{\Div}$
\begin{align}\label{ch:Stokes ass:growthweak}
\norm{G(v)}_{L_2(U;L^2_x)}^2 = \sum_{j\in \mathbb{N}} \norm{g_j\big(\cdot,v(\cdot)\big)}_{L^2_x}^2 \leq c_{\text{growth}}(1+\norm{v}_{L^2_x}^2),
\end{align} 
\item (Lipschitz continuity) for all $v_1, \, v_2 \in L^2_{\Div}$ it holds
\begin{align} \label{ch:Stokes ass:Lipschitzweak}
\norm{G(v_1) - G(v_2)}_{L_2(U;L^2_x)}^2 \leq c_{\text{lip}} \norm{u_1 - u_2}_{L^2_x}^2.
\end{align} 
\end{enumerate}
\end{assumption}

Assumption~\ref{ch:Stokes ass:NoiseCoeffweak} is standard for the derivation of weak solutions. Sometimes one couples the condition on the noise coefficient and the dissipation of the monotone operator $S$. Since we are only interested in regularity of $\pi$, we do not proceed this way.

\subsection{Existence of weak solutions}
The general theory of monotone SPDEs of Liu and Röckner~\cite{LiRo} covers the construction of the velocity variable. A similar construction for weak solutions to power-law fluids has been done in~\cite{Breit2015}. For the sake of completeness, we state the result and comment on the main ingredients.

\begin{theorem} \label{thm:ExistenceWeek}
Let $p \in (1,\infty)$, $q \in [1,\infty)$, $\mathcal{O}$ be open, bounded, Assumption~\ref{ch:Stokes ass:NoiseCoeffweak} be satisfied and $u_0 \in L^q_\omega L^2_{\Div}$ be $\mathcal{F}_0$-measurable. \\
Then there exists $u \in L^q_\omega C^0_t L^2_{\Div} \cap L^{p q/2}_\omega L^{p q/2}_t W^{1,p}_{0,x} $ adapted to $(\mathcal{F}_t)$ such that $u$ satisfies~\eqref{ch:Stokes eq:p-Stokes_weak} and 
\begin{align}\label{eq:L2Ito}
\mathbb{E}\left[ \sup_{t\in I} \norm{u(t)}_{L^2_x}^q + \left( \int_I \mathcal{J}\big(u(t) \big) \dd t \right)^{q/2} \right] \lesssim \norm{u_0}_{L^q_\omega L^2_x}^q + 1.
\end{align}
\end{theorem}

The abstract conditions postulated by Liu and R\"ockner, that guarantee the existence of a weak solution, can be verified along the lines of~\cite[Example 4.1.9]{MR3410409}. In fact, \cite[Example 4.1.9]{MR3410409} only deals with the case $p \geq 2$. However, the general case $p \in (1,\infty)$ can be recovered, if we adjust the underlying Gelfand triple to
\begin{align*}
W^{1,p}_{0,x} \cap L^2_{\Div} \hookrightarrow L^2_{\Div} \hookrightarrow \big( W^{1,p}_{0,x} \cap L^2_{\Div} \big)'.
\end{align*}

The key tools for the a priori bound~\eqref{eq:L2Ito} are It\^o's formula for $F(u) = \norm{u}_{L^2_x}^q$ together with a Gronwall argument. While the case $q=2$ is standard, the general case needs some minor changes. Since this is not the main focus of the article, we only refer to~\cite[Theorem~4.3.1]{wich_dis} for more details. 
%
%Due to the strong convexity~\eqref{eq:StronglyConvex} and $\mathcal{J}(0) = 0$, we find
%\begin{align*}
%\mathcal{J}(u) \lesssim D\mathcal{J}(u)[u].
%\end{align*}
%Therefore, the energy naturally provides a lower bound for the diffusion. 

\begin{remark}
The integrability in probability of the solution $u$ is purely determined by the integrability in probability of the initial condition $u_0$, i.e., for all $q \in [1,\infty)$
\begin{align*}
u_0 \in L^q_\omega L^2_x \quad \Rightarrow \quad  u \in L^q_\omega C^0_t L^2_x.
\end{align*} 
The limit case $q = \infty$ needs to be excluded, since the moment transfer is hindered by the Wiener process~$W$. The most one can hope for is $u \in L^{\Phi_2}_\omega L^2_x$, since $W \in L_\omega^{\Phi_2} \backslash L^\infty_\omega $. More details about integrability of Brownian motions in Banach spaces can be found in~\cite{VerHyt08}.
\end{remark}

\subsection{Reconstruction of the pressure} \label{sec:Pressure}
Thanks to the Helmholtz-Leray projection~$\Pi_{\Div}$ it is possible to target the construction of the velocity~$u$ and the pressure~$\pi$ individually. The pressure corresponds to the residual error when testing~\eqref{ch:Stokes eq:p-Stokes_weak} by smooth gradients rather than smooth divergence free fields, i.e., we use~\eqref{ch:Stokes eq:p-Stokes_weakPressure} as a definition of the pressure~$\pi$. Additionally, we artificially decompose $\pi = \pi_{\mathrm{det}} + \pi_{\mathrm{sto}} $ into the components
\begin{subequations}
\begin{align} \label{def:piDet}
\nabla \pi_{\mathrm{det}} &= \Pi_{\mathrm{div}}^{\perp} \Div S(\varepsilon u), \\ \label{def:piSto}
\nabla \pi_{\mathrm{sto}} &= \Pi_{\mathrm{div}}^{\perp} G(u) \dot{W},
\end{align} 
\end{subequations}
and discuss them separately.

First, we have a look at the deterministic component.
\begin{lemma} \label{lem:PressureDet}
Let $\mathcal{O}$ be a bounded $C^2$ domain and $ u \in W^{1,p}_{0,x}$. Then there exists $\pi_{\mathrm{det}} \in L^{p'}_{0,x}$ such that for all $\xi \in C^\infty_{0,x}$
\begin{align} \label{ch:Stokes eq:PressureDetFinal}
\int_{\mathcal{O}} \pi_\mathrm{det} \Div \xi \dd x = \int_{\mathcal{O}}  S(\varepsilon u): \nabla \Pi_{\Div}^\perp \xi \dd x 
\end{align}
and
\begin{align} \label{eq:estimatePiDet}
\norm{\pi_{\mathrm{det}}}_{L^{p'}_x} \lesssim \norm{S(\varepsilon u)}_{L^{p'}_x}.
\end{align}
\end{lemma}
\begin{proof}
We interpret~\eqref{def:piDet} distributionally, i.e., for $\xi \in C^\infty_{0,x}$
\begin{align} \label{ch:Stokes eq:PressureIdentify}
\langle \pi_{\mathrm{det}}, \Div \xi \rangle :=  \int_{\mathcal{O}}  S(\varepsilon u): \nabla \Pi_{\Div}^\perp \xi \dd x .
\end{align}
Without loss of generality we assume $\mean{\pi_{\mathrm{det}}}_{\mathcal{O}} = 0$.

Since we want to identify $\pi_{\mathrm{det}}$ as a proper function, we substitute $\xi = \mathcal{B} \xi_0$, where $\mathcal{B}$ denotes the Bogovskii operator (see Theorem~\ref{app:Aux thm:Bogovskii}) and $\xi_0 \in C^\infty_{0,x}$ with $\mean{\xi_0}_{\mathcal{O}} = 0$, to obtain
\begin{align} \label{eq:001}
\langle \pi_{\mathrm{det}}, \xi_0 \rangle = \int_{\mathcal{O}}  S(\varepsilon u): \nabla \Pi_{\Div}^\perp \mathcal{B} \xi_0 \dd x .
\end{align}
Thus, using the symmetry of $\Pi_{\Div}^\perp$ and integration by parts,~\eqref{eq:001} is equivalent to
\begin{align} \label{eq:002}
\pi_{\mathrm{det}} = -\mathcal{B}^* \Pi_{\Div}^\perp \Div S(\varepsilon u),
\end{align}
where $\mathcal{B}^*$ is the adjoint operator of $\mathcal{B}$. 

It remains to verify the boundedness of 
\begin{align*}
A:= -\mathcal{B}^* \Pi_{\Div}^\perp \Div : L^{p'}_{0,x} \to L^{p'}_{0,x} .
\end{align*}
We instead address the boundedness of the adjoint operator and use a duality argument.

The right hand side of~\eqref{eq:001} is estimated by H\"older's inequality
\begin{align*}
\int_{\mathcal{O}}  S(\varepsilon u): \nabla \Pi_{\Div}^\perp \mathcal{B} \xi_0 \dd x \leq  \norm{S(\varepsilon u)}_{L_x^{p'}}  \,\bignorm{ \nabla \Pi_{\Div}^\perp \mathcal{B} \xi_0 }_{L^p_x}.
\end{align*}
We further estimate, due to the Sobolev stability of the Helmholtz-Leray projection (Theorem~\ref{app:FuncSpace thm:GradientStabilityHelm}) and the Bogovskii operator (Theorem~\ref{app:Aux thm:Bogovskii} with $s = 0$, $q = p$), 
\begin{align*}
\bignorm{ \nabla \Pi_{\Div}^\perp \mathcal{B} \xi_0 }_{L^p_x} \lesssim \bignorm{ \nabla \mathcal{B} \xi_0 }_{L^p_x} \lesssim \bignorm{\xi_0 }_{L^p_x}.
\end{align*}
Therefore, the linear operator $A : L^p_{0,x} \to L^p_{0,x}$ is bounded.

The claim~\eqref{ch:Stokes eq:PressureDetFinal} follows by our construction. The inequality~\eqref{eq:estimatePiDet} is verified by a density argument
\begin{align*}
\norm{\pi_{\mathrm{det}}}_{L^{p'}_{x}} = \sup_{ \xi \in L^p_{0,x}} \frac{\langle \pi_{\mathrm{det}}, \xi \rangle}{\norm{\xi}_{L^p_x}} =  \sup_{ \xi \in C^\infty_{0,x}, \mean{\xi}_{\mathcal{O}} = 0 } \frac{\langle \pi_{\mathrm{det}}, \xi \rangle}{\norm{\xi}_{L^p_x}} \lesssim \norm{S(\varepsilon u)}_{L_x^{p'}}.
\end{align*}
\end{proof}

Second, we investigate the stochastic pressure. The main difficulty is the limited time regularity. Therefore, we initially take a look at the time integrated pressure.

\begin{lemma} \label{lem:integratedPressure}
Let $\mathcal{O}$ be a bounded domain with locally Lipschitz boundary and $\mathcal{I}_W\big(G(u) \big) \in L^2_\omega B^{1/2}_{\Phi_2,\infty} L^2_x$. Then there exists $\mathcal{K}_{\mathrm{sto}} \in L^2_\omega B^{1/2}_{\Phi_2,\infty} \big( W^{1,2}_{0,x} \cap L^2_{0,x}\big) $ such that for all $t \in I$ and $\xi \in C^\infty_{0,x}$ it holds
\begin{align} \label{eq:Ksto}
\int_{\mathcal{O}} \mathcal{K}_{\mathrm{sto}}(t) \Div \xi \dd x &= - \int_{\mathcal{O}}  \mathcal{I}_W \big( G(u) \big)(t) \Pi_{\Div}^\perp \xi \dd x.
\end{align}
Moreover,
\begin{align} \label{eq:Estimate01}
\norm{\mathcal{K}_{\mathrm{sto}}}_{L^2_\omega B^{1/2}_{\Phi_2,\infty} W^{1,2}_x} \lesssim \norm{\mathcal{I}_W \big( G(u) \big)}_{L^2_\omega B^{1/2}_{\Phi_2,\infty} L^2_x}.
\end{align}
\end{lemma}
\begin{proof}
The proof proceeds similar to the one of Lemma~\ref{lem:PressureDet}.

First, note~\eqref{eq:Ksto} is equivalent to
\begin{align} 
\mathcal{K}_{\mathrm{sto}}(t) := -\mathcal{B}^* \Pi_{\Div}^\perp \mathcal{I}_W\big( G(u) \big)(t),
\end{align}
where $\mathcal{B}^*$ is the adjoint of the Bogovskii operator. In particular, $\mathcal{K}_{\mathrm{sto}}$ is mean-value free.

Due to Theorem~\ref{app:Aux thm:Bogovskii}, one can check that $ \mathcal{B}^*: L^2_{0,x} \to W^{1,2}_{0,x} \cap L^2_{0,x}$ is bounded. Indeed, let $\xi, \zeta \in C^\infty_{0,x}$ such that $\mean{\xi}_{\mathcal{O}} = \mean{\zeta}_{\mathcal{O}} = 0$, then by H\"older's inequality and Theorem~\ref{app:Aux thm:Bogovskii} with $s = -1$ and $q = 2$,
\begin{align*}
\int_{\mathcal{O}} \mathcal{B}^* \xi  \zeta \dd x = \int_{\mathcal{O}}  \xi \cdot \mathcal{B} \zeta \dd x \leq \norm{\xi}_{L^2_x} \norm{\mathcal{B} \zeta}_{L^2_x} \lesssim \norm{\xi}_{L^2_x} \norm{ \zeta}_{W^{-1,2}_{x}}.
\end{align*}
Thus, using the density of smooth mean-value free functions within $W^{1,2}_{0,x} \cap L^2_{0,x}$,
\begin{align} \label{eq:0A}
\norm{\mathcal{B}^* \xi }_{W^{1,2}_{x}} \lesssim \norm{\xi}_{L^2_x}.
\end{align}

Finally,~\eqref{eq:0A} together with $L^2_x$-stability of the Helmholtz-Leray projection establish
\begin{align*}
\norm{\mathcal{K}_{\mathrm{sto}}(t)}_{W^{1,2}_x} \lesssim \norm{\Pi_{\Div}^\perp  \mathcal{I}_W \big( G(u) \big)(t) }_{L^2_x} \leq \norm{\mathcal{I}_W \big( G(u) \big)(t) }_{L^2_x}.
\end{align*}
An application of the $L^2_\omega B^{1/2}_{\Phi_2,\infty}$-norm verifies~\eqref{eq:Estimate01}. 
\end{proof}

Now, we have identified the regularity class for the integrated pressure. The next step is the transfer of the results to the pressure. For a smooth test function $\xi \in C^\infty_0\big( [0,T); C^\infty_{0,x} \big)$ we define the distribution
\begin{align}\label{eq:pisto}
\langle \pi_{\mathrm{sto}}, \xi \rangle := -\int_I \langle \mathcal{K}_{\mathrm{sto}}, \partial_t \xi \rangle \dd t.
\end{align}
To shorten the notation we abbreviate $X = W^{1,2}_{0,x} \cap L^2_{0,x}$.

\begin{lemma} \label{lem:pressureEstimate}
Let $\alpha \in (0,1)$ and $p,q \in (1,\infty)$. Moreover, assume that $\mathcal{K}_{\mathrm{sto}} \in L^2_\omega B^{\alpha}_{p,q} X$. Then $\pi_{\mathrm{sto}} \in L^2_\omega B^{\alpha-1}_{p,q} X$ and
\begin{align} \label{eq:pressureEstimate}
\norm{\pi_{\mathrm{sto}}}_{L^2_\omega B^{\alpha -1}_{p,q} X} \lesssim \norm{\mathcal{K}_{\mathrm{sto}}}_{L^2_\omega B^{\alpha}_{p,q} X}.
\end{align}
\end{lemma}

\begin{proof}
Let $\xi \in C^\infty_0\big( I; C^\infty_{0,x} \big)$ be spatially mean-value free. Due to~\eqref{eq:pisto} and duality
\begin{align} \label{eq:pressure01}
\langle \pi_{\mathrm{sto}},\xi \rangle  \leq \norm{\mathcal{K}_{\mathrm{sto}}}_{B^{\alpha}_{p,q} X} \norm{\partial_t \xi}_{(B^{\alpha}_{p,q} X)'}.
\end{align}

Note that $\xi$ trivially extends by zero to the full space in a smooth way. Now, we can use the equivalent spectral characterization of Besov norms, cf. Corollary~\ref{cor:spectralBesov}, 
\begin{align*}
\norm{\partial_t \xi}_{\big(B^{\alpha}_{p,q}(\mathbb{R}; X) \big)'} \eqsim \norm{\partial_t \xi}_{\big(\tilde{B}^{\alpha}_{p,q}(\mathbb{R}; X)\big)'}.
\end{align*}
Since $X$ is reflexive and $p, q < \infty$ we can identify the dual using Lemma~\ref{lem:dual}
\begin{align*}
\norm{\partial_t \xi}_{\big(\tilde{B}^{\alpha}_{p,q}(\mathbb{R}; X \big)'} \eqsim \norm{\partial_t \xi}_{\tilde{B}^{-\alpha}_{p',q'}(\mathbb{R}; X')}.
\end{align*}
An application of Lemma~\ref{app:lem_der} yields
\begin{align} \label{eq:key}
\norm{\partial_t \xi}_{\tilde{B}^{-\alpha}_{p',q'}(\mathbb{R}; X')} \lesssim \norm{\xi}_{\tilde{B}^{1-\alpha}_{p',q'}(\mathbb{R}; X')}.
\end{align}
Going back to the Besov norm in terms of integrability of differences, cf. Theorem~\ref{thm:spectralBesov}, and using that $\xi$ is extended by zero
\begin{align*}
\norm{\xi}_{\tilde{B}^{1-\alpha}_{p',q'}(\mathbb{R}; X')} \eqsim \norm{\xi}_{B^{1-\alpha}_{p',q'}(\mathbb{R}; X')} =  \norm{\xi}_{B^{1-\alpha}_{p',q'} X'}.
\end{align*}
This concludes
\begin{align} \label{eq:pressure02}
\norm{\partial_t \xi}_{\big(B^{\alpha}_{p,q}(\mathbb{R}; X) \big)'} \lesssim \norm{\xi}_{B^{1-\alpha}_{p',q'} X'}.
\end{align}

Overall, revisiting~\eqref{eq:pressure01} and using the density of smooth mean-value free functions in $X$ together with~\eqref{eq:pressure02},
\begin{align*}
\norm{\pi_{\mathrm{sto}}}_{(B^{1-\alpha}_{p',q'} X')'} = \sup_{\xi \in C^\infty_0( I; C^\infty_{0,x} ),\, \forall t:\mean{\xi(t)} = 0} \frac{\langle \pi_{\mathrm{sto}}, \xi \rangle}{\norm{\xi}_{B^{1-\alpha}_{p',q'}X}} \lesssim \norm{\mathcal{K}_{\mathrm{sto}}}_{B^{\alpha}_{p,q} X}.
\end{align*}
Ultimately, due to the equivalent norm on the dual space, cf. Lemma~\ref{lem:dual},
\begin{align*}
\norm{\pi_{\mathrm{sto}}}_{(B^{1-\alpha}_{p',q'} X')'} \eqsim \norm{\pi_{\mathrm{sto}}}_{B^{\alpha-1}_{p,q} X}.
\end{align*}
Taking the square and expectation verifies~\eqref{eq:pressureEstimate}.
\end{proof}

\begin{remark}
So far the stochastic pressure was always estimated on suboptimal spaces. For example in \cite{MR2004282,MR3022227} the authors infer $\pi_{\mathrm{sto}} \in L^1_\omega W^{-1,\infty}_t L^2_{0,x}$. They neglect all additional information on the temporal differentiability of $\mathcal{K}_{\mathrm{sto}}$ and only use the boundedness.

Whether a corresponding result of Lemma~\ref{lem:pressureEstimate} remains valid on the scale of exponentially integrable Besov spaces boils done to the understanding of the interaction between derivatives and norms in the spirit of~\eqref{eq:key}, i.e., 
\begin{align}
\norm{\partial_t \xi}_{B^{-\alpha}_{\Phi_2^*,1} W^{-1,2}_x} \lesssim \norm{\xi}_{B^{1-\alpha}_{\Phi_2^*,1} W^{-1,2}_x}.
\end{align}

In order to establish~\eqref{eq:key}, we heavily relied on the spectral representation of Besov norms. However, the equivalence of the spectral representation and the definition in terms of integrated differences fails for $p \in \{ 1, \infty \}$. Particularly, $\Phi_2(t) = e^{t^2}-1$ is to close to $\infty$ and whence $\Phi_2^*$ is to close to $1$. 

\end{remark}

Still, in regard of the regularity result of the integrated pressure, cf. Lemma~\ref{lem:integratedPressure}, it is natural to define the set 
\begin{align} \label{eq:NegativeBesov}
\begin{aligned}
&\mathbb{B}^{-\alpha}_{\Phi,r}(I; X) \\
&:= \left\{ u \in \big( C^\infty_0(I;X') \big)' \big| \, \exists g \in B^{1-\alpha}_{\Phi,r}(I;X) \forall \xi \in C^\infty_0(I;X') :\, \langle u,  \xi \rangle = -\int_I \langle g, \partial_t \xi \rangle \dd t   \right\}.
\end{aligned}
\end{align}

\begin{corollary}
In the setting of Lemma~\ref{lem:integratedPressure} we have $\pi_{\mathrm{sto}} \in \mathbb{B}^{-1/2}_{\Phi_2,\infty} X$ $\mathbb{P}$-a.s.
\end{corollary}
\begin{proof}
We only need to verify that the distribution $\pi_{\mathrm{sto}}$ defined by~\eqref{eq:pisto} is well-defined. The claim that $\pi_{\mathrm{sto}}$ belongs to $ \mathbb{B}^{-1/2}_{\Phi_2,\infty} X$ follows by the definition~\eqref{eq:NegativeBesov} and the assumption on $\mathcal{K}_{\mathrm{sto}}$.

Let $\xi \in C^\infty_0\big( I; X' \big)$. Due to duality and H\"older's inequality
\begin{align*}
\int_I \langle \mathcal{K}_{\mathrm{sto}}, \partial_t \xi \rangle \dd t &\leq \int_I \norm{\mathcal{K}_{\mathrm{sto}}}_{X}\norm{\partial_t \xi }_{X'} \dd t \\
&\lesssim \norm{\mathcal{K}_{\mathrm{sto}}}_{L^{\Phi_2}_t X} \norm{\partial_t \xi }_{L^{\Phi_2^*}_t X'}.
\end{align*}
Next, we take the square and expectation
\begin{align*}
\mathbb{E} \abs{ \langle \pi_{\mathrm{sto}}, \xi \rangle }^2 \lesssim \norm{\mathcal{K}_{\mathrm{sto}}}_{L^2_\omega L^{\Phi_2}_t X}^2 \norm{\partial_t \xi }_{L^{\Phi_2^*}_t X'}^2.
\end{align*}
Thus, we have shown that $\pi_{\mathrm{sto}}$ is well-defined.
\end{proof}

\subsection{Proof of Theorem~\ref{ch:Stokes thm:EstimateWeakSolution}}
The previous results lead to a simple proof of Theorem~\ref{ch:Stokes thm:EstimateWeakSolution}.

\begin{proof}[Proof of Theorem~\ref{ch:Stokes thm:EstimateWeakSolution}]

Due to the factorization~\eqref{ch:intro intro:p-Stokes-stoch00New} it is possible to construct the velocity field~$u$ independently of the pressure~$\pi$. Afterwards the pressure is reconstructed. 

The existence of the velocity field and verification of the a priori bound~\eqref{ch:Stokes eq:EstimateWeakSolution} is done in Theorem~\ref{thm:ExistenceWeek}.

In Section~\ref{sec:Pressure} we deal with the pressure reconstruction. 
Lemma~\ref{lem:PressureDet} verifies the regularity of the deterministic pressure. The stochastic pressure is handled in Lemma~\ref{lem:integratedPressure} and~\ref{lem:pressureEstimate}. In particular, Lemma~\ref{lem:integratedPressure} ensures that $\mathcal{K}_{\mathrm{sto}} \in L^2_\omega B^{1/2}_{\Phi_2,\infty} X \hookrightarrow L^2_\omega B^{\alpha}_{s,r} X$ for any $\alpha \in (0,1/2)$ and $s,r \in (1,\infty)$. Therefore, Lemma~\ref{lem:pressureEstimate} implies~\eqref{ch:Stokes eq:EstimateWeakSolution03}.
\end{proof}

\section{Existence of strong solutions} \label{sec:ExistenceStrong}

Within this section we discuss the existence of strong solutions to~\eqref{ch:Stokes intro:p-Stokes-stoch}. The main difference between weak and strong solutions lies in the fact that the latter can be interpreted in a point-wise manner, i.e., there is no need to interpret~\eqref{ch:Stokes intro:p-Stokes-stoch01} in a distributional sense.

 Strong solutions for a related model have already been constructed in~\cite{MR4022286} by means of improved spatial regularity. They formally test~\eqref{ch:Stokes intro:p-Stokes-stoch} by the Laplacian. 

We on the other hand rely on the gradient flow structure~\eqref{ch:Stokes intro:p-Stokes-stoch01New}, which -- at least formally -- corresponds to a test with $-\Pi_{\Div}\Div S(\varepsilon u)$.

\subsection{Strong solutions}
\begin{definition} \label{def:strong}
Let $u_0 \in L^2_\omega L^2_{\Div}$ be $\mathcal{F}_0$-measurable. The tupel $(u,\pi)$ is called strong solution if it is a weak solution and additionally satisfies
\begin{enumerate}
\item $\Pi_{\Div} \Div S(\varepsilon u) \in L^2_\omega L^2_t L^2_{\Div}$,
\item for all $t \in I$ and $\mathbb{P}-a.s.$ it holds
\begin{align} \label{ch:Stokes eq:StrongFormulation}
u(t) - u_0 - \int_0^t \Pi_{\Div} \Div S(\varepsilon u) \ds = \sum_{j\in \mathbb{N}} \int_0^t \Pi_{\Div} g_j(\cdot,u) \dd \beta^j(s)
\end{align}
as an equation in $L^2_{\Div}$.
\end{enumerate}
\end{definition}

In general, strong solutions only upgrade the regularity of the divergence-free component of the equation. In particular, the pressure does not enjoy higher regularity for strong solutions compared to weak solutions.

\subsection{A sufficient condition for strong solutions}

\begin{assumption}\label{ch:Stokes ass:NoiseCoeffstrong}
We assume that $\set{g_j}_{j\in \mathbb{N}} \in C^1(\mathcal{O} \times \mathbb{R}^n; \mathbb{R}^n)$ such that for all $v \in W^{1,p}_{0,x} \cap L^2_{\Div}$ it holds
\begin{align}\label{ch:Stokes ass:growthstrong}
\norm{\varepsilon \Pi_{\Div} G(v) }_{L_2(U;L^p_x)}^2 = \sum_{j\in \mathbb{N}} \norm{\varepsilon \Pi_{\Div} g_j\big(\cdot,v(\cdot)\big)}_{L^p_x}^2 \lesssim (1+\norm{\varepsilon v}_{L^p_x}^2).
\end{align}
\end{assumption}

Assumption~\ref{ch:Stokes ass:NoiseCoeffstrong} can be verified by using the Sobolev stability of $\Pi_{\Div}$, cf. Theorem~\ref{app:FuncSpace thm:GradientStabilityHelm}, and a suitable assumption on the data $\set{g_j}$.

\subsection{Proof of Theorem~\ref{ch:Stokes thm:StrongEnergy}}

Our main tool in the derivation of energy bounds for strong solutions is the gradient flow structure of the equation.

\begin{proof}[Proof of Theorem~\ref{ch:Stokes thm:StrongEnergy}]
We will only address the case $q = 1$. General moments can be obtained by expanding the energy $\mathcal{J}(u)^{q}$. For more details we refer to~\cite[Theorem~4.3.1]{wich_dis}. Our aim is to apply the result~\cite[Theorem~1.4]{Gess2012Strong}. Therefore, we need to check the assumptions~$($A1$)-($A6$)$.

First of all we relate our framework to the one used in~\cite{Gess2012Strong}. We choose $H = L^2_{\Div}$, $S = W^{1,p}_{0,x}$ and $V = H \cap S$. Additionally, we let $\varphi(v) = \tilde{\varphi}(v) = \mathcal{J}(v)$ and $B_t(v) \equiv \Pi_{\Div} G(v)$. Keep in mind that we use $\varphi$ as the potential that defines the energy whereas Gess uses $\varphi$ to denote the energy itself.

Ad $($A1$)$: First we show that $\mathcal{J} : W^{1,p}_{0,x} \to \mathbb{R}$ is continuous. Indeed, let $u,v \in W^{1,p}_{0,x}$. Then, due to the fundamental theorem,
\begin{align*}
\abs{\mathcal{J}(u) - \mathcal{J}(v)} &= \abs{\int_{\mathcal{O}} \varphi_{\kappa}(\abs{\varepsilon u}) -  \varphi_{\kappa}(\abs{\varepsilon u}) \dd x} \\
&= \abs{\int_{\mathcal{O}} \int_0^1 \varphi_{\kappa}'( \abs{\varepsilon w_\theta}) \frac{\varepsilon w_\theta }{\abs{\varepsilon w_\theta}}: \varepsilon(u-v) \dd \theta \dd x} \\
&= \abs{\int_{\mathcal{O}} \int_0^1 (\kappa + \abs{\varepsilon w_\theta})^{p-2} \varepsilon w_\theta: \varepsilon(u-v) \dd \theta \dd x} \\
&\leq  \int_{\mathcal{O}} \int_0^1 (\kappa + \abs{\varepsilon w_\theta})^{p-1}  \dd \theta \abs{\varepsilon(u-v)} \dd x, 
\end{align*}
where $w_\theta= v + \theta(u-v)$. Note that $\abs{\varepsilon w_\theta} \leq \max\{ \abs{\varepsilon u}, \abs{\varepsilon v}\}$ for all $\theta \in [0,1]$. This and H\"older's inequality imply
\begin{align*}
 &\int_{\mathcal{O}} \int_0^1 (\kappa + \abs{\varepsilon w_\theta})^{p-1}  \dd \theta \abs{\varepsilon(u-v)} \dd x \\
 & \quad \leq  \left( \int_{\mathcal{O}} (\kappa +  \max\{ \abs{\varepsilon u}, \abs{\varepsilon v}\})^{p} \dd x \right)^{(p-1)/p} \left( \int_{\mathcal{O}}  \abs{\varepsilon(u-v)}^p \dd x \right)^{1/p}.  
\end{align*}
Finally, we arrive at
\begin{align*}
\abs{\mathcal{J}(u) - \mathcal{J}(v)} \leq \left(\int_{\mathcal{O}} (\kappa + \abs{\varepsilon u})^p \dd x  + \int_{\mathcal{O}} (\kappa + \abs{\varepsilon v})^p  \dd x \right)^{(p-1)/p} \norm{\nabla(u-v)}_{L^p_x}.
\end{align*}
The continuity is verified. 

Moreover, using that $\varphi'(t) = t^{p-1}$ is homogeneous of degree $p-1$ and a substitution,
\begin{align*}
\varphi_{\kappa}(2t) &= \int_0^{2t} \frac{\varphi'(\kappa + s)}{\kappa + s } s \dd s = \int_0^{2t} 2^{p-1} \frac{\varphi'(\frac{\kappa}{2} + \frac{s}{2})}{\frac{\kappa}{2} + \frac{s}{2} } \frac{s}{2} \dd s \\
&=2^p \int_0^{t}  \frac{\varphi'(\frac{\kappa}{2} + s)}{\frac{\kappa}{2} + s } s \dd s = 2^p \varphi_{\frac{\kappa}{2}}(t).
\end{align*}
Since $p \geq 2$ we find $\varphi_{\kappa}(t)$ is increasing in $\kappa$ for fixed $t$ and whence $\mathcal{J}(2u) \leq 2^p \mathcal{J}(u)$.

Ad $($A2$)$: Recall
\begin{align*}
\varphi_{\kappa}(t) &= \frac{1}{p}(\kappa + t)^p - \frac{\kappa}{p-1}(\kappa + t)^{p-1} + \frac{\kappa^p}{p(p-1)}, \\
\varphi_{\kappa}'(t) &=(\kappa + t)^{p-2} t, \\
\varphi_{\kappa}''(t) &= (\kappa +  t)^{p-2}\left( 1 + (p-2) \frac{t}{\kappa + t}\right).
\end{align*} 

Applying a weighted Young's inequality one finds
\begin{align} \label{eq:power}
 \frac{1}{2p}(\kappa + t)^p - \frac{2^{p-1} - 1}{p(p-1)}\kappa^p \leq \varphi_{\kappa}(t) \leq  \frac{1}{p}(\kappa + t)^p  + \frac{1}{p(p-1)}\kappa^p.
\end{align}
In other words $\varphi_{\kappa}$ behaves like the $p$th-power shifted along $\kappa$. It follows using the monotonicity of powers and~\eqref{eq:power}
\begin{align} \label{eq:EquivalenceShift}
&\int_{\mathcal{O}} \frac{1}{p}\abs{\varepsilon u}^p \dd x \leq \int_{\mathcal{O}}\frac{1}{p} ( \kappa +  \abs{ \varepsilon u})^p  \dd x \lesssim \int_{\mathcal{O}} \varphi_{\kappa}(\abs{\varepsilon u}) + \kappa^p \abs{\mathcal{O}}\eqsim \mathcal{J}(u) + 1.
\end{align}

Let $\set{w_k}_{k \in \mathbb{N}} \subset W^{1,p}_{0,x}$. Now, using~\eqref{eq:2ndGateaux}, H\"older's and Young's inequalities and~\eqref{eq:EquivalenceShift}
\begin{align*}
\sum_{k=1}^\infty D^2 \mathcal{J}(v)[w_k,w_k] &\leq (p-1) \sum_{k=1}^\infty \int_{\mathcal{O}} (\kappa + \abs{\varepsilon v})^{p-2} \abs{\varepsilon w_k}^2 \dd x \\
&\leq (p-1) \left( \int_{\mathcal{O}} (\kappa + \abs{\varepsilon v})^{p} \dd x \right)^{(p-2)/p} \sum_{k=1}^\infty \norm{\varepsilon w_k}_{L^p_x}^2  \\
&\leq  \frac{(p-1)p}{p-2} \int (\kappa + \abs{\varepsilon v})^p \dd x + \frac{(p-1)2}{p} \left( \sum_{k=1}^\infty \norm{\varepsilon w_k}_{L^p_x}^2 \right)^{p/2} \\
&\lesssim \mathcal{J}(v) + 1 + \left( \sum_{k=1}^\infty \norm{\varepsilon w_k}_{L^p_x}^2 \right)^{p/2}.
\end{align*}

Ad $($A3$)$: Lemma~\ref{lem:Taylor} shows that $\mathcal{J}$ is strongly convex.

Ad $($A4$)$: Recall~\eqref{eq:1stGateaux}. Therefore,
\begin{align*}
-D\mathcal{J}(v)[v] = -\int_{\mathcal{O}} \varphi_{\kappa}'(\abs{\varepsilon v}) \abs{\varepsilon v} \dd x \eqsim - \int_{\mathcal{O}} \varphi_{\kappa}(\abs{ \varepsilon v}) \dd x = - \mathcal{J}(v) \leq 0.
\end{align*}

Ad $($A5$)$: Follows by the assumption on the noise coefficient, cf.~\eqref{ch:Stokes ass:Lipschitzweak}.

Ad $($A6$)$: The major ingredient that enables strong solutions is the regularity of the noise. Recall~\eqref{ch:Reg def:NoiseCoeff} and~$($A2$)$ with $w_k = \Pi_{\Div}G(v)u_k = \Pi_{\Div} g_k(\cdot, v )$. It remains to estimate
\begin{align*}
\left( \sum_{k=1}^\infty \norm{\varepsilon w_k}_{L^p_x}^2 \right)^{p/2} = \left( \sum_{k=1}^\infty \norm{\varepsilon\Pi_{\Div} g_k(\cdot, v )}_{L^p_x}^2 \right)^{p/2} =\norm{\varepsilon \Pi_{\Div} G(v)}_{L_2(U;L^p_x)}^p.
\end{align*}
An application of the Assumption~\ref{ch:Stokes ass:NoiseCoeffstrong} and~\eqref{eq:EquivalenceShift} show
\begin{align*}
\norm{\varepsilon \Pi_{\Div} G(v)}_{L_2(U;L^p_x)}^p \lesssim  \norm{\varepsilon v}_{L^p_x}^p + 1 \lesssim \mathcal{J}(v) + 1.
\end{align*}

Overall, we have verified $($A1$)-($A6$)$ and therefore we may apply~\cite[Theorem~1.4]{Gess2012Strong} which provides us with the a priori bound
\begin{align*}
\sup_{t \in I} \mathbb{E} \left[ \mathcal{J}(u) \right] + \mathbb{E } \left[ \int_I \norm{\Pi_{\Div} \Div S(\varepsilon u)}_{L^2_x}^2 \dd s \right] \lesssim \mathbb{E} \left[ \mathcal{J}(u_0) \right] + 1.
\end{align*}

A careful consideration of the proof shows that one can also obtain an estimate where the supremum in time is taken before the expectation. One additional term, due to the stochastic integral, needs to be considered, i.e.,
\begin{align*}
&\mathbb{E} \left[ \sup_{t\in I} \int_0^t \sum_{k} D \mathcal{J}(u)[\Pi_{\Div} g_k(\cdot, v)] \dd \beta(s) \right] \\
&\quad = \mathbb{E} \left[ \sup_{t\in I} \int_0^t \sum_{k} \int_{\mathcal{O}} (\kappa + \abs{\varepsilon u})^{p-2} \varepsilon u: \varepsilon \Pi_{\Div} g_k(\cdot, v)] \dd \beta^k(s) \right] .
\end{align*}

Due to the Burkholder-Davis-Gundy inequality
\begin{align*}
&\mathbb{E} \left[ \sup_{t\in I} \int_0^t \sum_{k} \int_{\mathcal{O}} (\kappa + \abs{\varepsilon u})^{p-2} \varepsilon u: \varepsilon \Pi_{\Div} g_k(\cdot, v)] \dd x \dd \beta^k(s) \right] \\
&\quad \eqsim \mathbb{E} \left[ \left( \int_I \sum_{k} \left( \int_{\mathcal{O}} (\kappa + \abs{\varepsilon u})^{p-2} \varepsilon u: \varepsilon \Pi_{\Div} g_k(\cdot, v)] \dd x \right)^2 \dd s \right)^{1/2} \right].
\end{align*}
Invoking H\"older's inequality
\begin{align*}
&\mathbb{E} \left[ \left( \int_I \sum_{k} \left( \int_{\mathcal{O}} (\kappa + \abs{\varepsilon u})^{p-2} \varepsilon u: \varepsilon \Pi_{\Div} g_k(\cdot, v)] \dd x \right)^2 \dd s \right)^{1/2} \right] \\
&\quad \leq \mathbb{E} \left[ \left( \int_I  \left( \int_{\mathcal{O}} (\kappa + \abs{\varepsilon u})^{p} \dd x \right)^{2(p-1)/p} \norm{\varepsilon \Pi_{\Div} G(v)}_{L_2(U;L^p_x)}^2 \dd s \right)^{1/2} \right]. 
\end{align*}
The Assumption~\ref{ch:Stokes ass:NoiseCoeffstrong} together with~\eqref{eq:EquivalenceShift} show
\begin{align*}
&\mathbb{E} \left[ \left( \int_I  \left( \int_{\mathcal{O}} (\kappa + \abs{\varepsilon u})^{p} \dd x \right)^{2(p-1)/p} \norm{\varepsilon \Pi_{\Div} G(v)}_{L_2(U;L^p_x)}^2 \dd s \right)^{1/2} \right] \\
&\quad \lesssim \mathbb{E} \left[ \left( \int_I \mathcal{J}(u)^2+1 \dd s \right)^{1/2} \right]  \lesssim \mathbb{E} \left[ \left( \int_I \mathcal{J}(u)^2 \dd s \right)^{1/2} \right] + 1.
\end{align*}
Finally, Young's inequality verifies
\begin{align*}
&\mathbb{E} \left[ \sup_{t\in I} \int_0^t \sum_{k} D \mathcal{J}(u)[\Pi_{\Div} g_k(\cdot, v)] \dd \beta(s) \right] \\
&\quad \leq \delta \mathbb{E} \left[\sup_{t \in I} \mathcal{J}(u) \right] + c_\delta \left( \mathbb{E} \left[ \int_I \mathcal{J}(u) \dd s \right] + 1 \right).
\end{align*}
This allows for a Gronwall type argument that establishes the improved a priori estimate
\begin{align*}
\mathbb{E} \left[ \sup_{t \in I} \mathcal{J}(u) \right] +\mathbb{E} \left[ \int_I \norm{\Pi_{\Div} \Div S(\varepsilon u)}_{L^2_x}^2 \dd s \right] \lesssim \mathbb{E}\left[ \mathcal{J}(u_0) \right] + 1.
\end{align*}
 
\end{proof}

\begin{remark}
On a first view we could ask, if one could expand the $q$-energy of the symmetric $p$-Stokes system. However, additional difficulties arise to to the presence of the symmetric gradient. In particular, one needs to understand the interaction between projected symmetric $p$- and $q$-Laplacian. If one could establish a sign
\begin{align*}
\int_{\mathcal{O}} \Pi_{\Div} \Div \big(\abs{\varepsilon u}^{p-2} \varepsilon u \big) \cdot \Pi_{\Div} \Div \big( \abs{\varepsilon u}^{q-2} \varepsilon u \big) \dd x \geq 0,
\end{align*}
then it is possible to obtain improved gradient regularity.

Regularity questions about $\Div S(\varepsilon u)$ -- and ultimately on $\nabla S(\varepsilon u)$ -- are out of reach, since
\begin{align*}
\Pi_{\Div}^\perp \Div S(\varepsilon u) = \nabla \pi_{\mathrm{det}} \in L^{p'}_\omega L^{p'}_t W^{-1,p'}_x
\end{align*}
is only a distribution. The derivation of higher regularity for the pressure is a non-trivial task.
\end{remark}

\section{Temporal regularity of strong solutions} \label{sec:RegularityStrong}
Fortunately, the strong formulation of equation~\eqref{ch:Stokes intro:p-Stokes-stoch} is sufficient to derive improved temporal regularity. We already established a similar result for the $p$-Laplace system in~\cite{wichmann2021temporal}.

\subsection{Exponential Besov-regularity}
In general, it is delicate to derive temporal regularity on the limited threshold of order $1/2$ for stochastic partial differential equations. A key example was derived by Hyt\"{o}nen and Veraar in~\cite{VerHyt08}. They show that $\mathbb{P}$-almost surely Wiener processes belong to the Besov space $B^{1/2}_{\Phi_2,\infty}$ and do not belong to the Besov space $B^{1/2}_{p,q}$ for $p,q < \infty$. 

We show that strong solutions inherit similar temporal regularity properties from the driving Wiener process.

\begin{proof}[Proof of Theorem~\ref{ch:Stokes thm:BesovReg}]
We will only present part~\ref{it:b} of Theorem~\ref{ch:Stokes thm:BesovReg} in detail and comment on the necessary changes for part~\ref{it:a}.

We start with the estimate for the Besov-Orlicz semi-norm
\begin{align*}
\mathbb{E} \left[ \seminorm{u}_{B^{1/2}_{\Phi_2,\infty}(I;L^2_x)}^2 \right] = \mathbb{E} \left[ \left( \sup_h h^{-1/2} \norm{\tau_h u}_{L^{\Phi_2}(I \cap I-\set{h}; L^2_x)} \right)^2 \right].
\end{align*}
Using the strong formulation~\eqref{ch:Stokes eq:StrongFormulation}
\begin{align*}
\mathbb{E} \left[ \seminorm{u}_{B^{1/2}_{\Phi_2,\infty}(I;L^2_x)}^2 \right] &\lesssim \mathbb{E} \left[ \left( \sup_h h^{-1/2} \norm{\int_{\cdot}^{\cdot + h} \Pi_{\Div} \Div S(\varepsilon u) \ds  }_{L^{\Phi_2}(I \cap I-\set{h}; L^2_x)} \right)^2 \right] \\
&\quad + \mathbb{E} \left[ \left( \sup_h h^{-1/2} \norm{\tau_h \mathcal{I}(G(u))}_{L^{\Phi_2}(I \cap I-\set{h}; L^2_x)}  \right)^2 \right] \\
&=: \mathrm{I} + \mathrm{II}.
\end{align*}

Invoking $L^{\infty} \hookrightarrow L^{\Phi_2}$ and H\"older's inequality
\begin{align*}
\mathrm{I} &\lesssim \mathbb{E} \left[ \left( \sup_h h^{-1/2} \sup_{t \in I \cap I-\set{h}} \int_{t}^{t + h} \norm{\Pi_{\Div} \Div S(\varepsilon u)  }_{L^2_x} \ds \right)^2 \right] \\
&\leq \mathbb{E} \left[  \sup_h \sup_{t \in I \cap I-\set{h}} \int_{t}^{t + h} \norm{\Pi_{\Div} \Div S(\varepsilon u)  }_{L^2_x}^2 \ds  \right] \\
&\leq \norm{\Pi_{\Div} \Div S(\varepsilon u)}_{L^2_\omega L^2_t L^2_x}^2 .
\end{align*}

The second term is bounded by Theorem~\ref{app:StochInt thm:Stability}~\ref{app:StochInt it:5}, ~\eqref{app:StochInt eq:RelGammaL2} and~\eqref{ch:Stokes ass:growthweak}
\begin{align*}
\mathrm{II} \leq \norm{\mathcal{I}(G(u))}_{L^{2}_\omega B_{\Phi_2,\infty}^{1/2}(0,T;L^2_x)}^2  &\lesssim \norm{G(u)}_{L^{N_2}_\omega L^\infty_t \gamma(U;L^2_x)}^2 \\
&= \norm{G(u)}_{L^{N_2}_\omega L^\infty_t L_2(U;L^2_x)}^2 \lesssim \norm{u}_{L^{N_2}_\omega L^\infty_t L^2_x}^2 + 1.
\end{align*}
This establishes the semi-norm estimate. 

The full norm follows by the embedding $L^\infty \hookrightarrow L^{\Phi_2}$
\begin{align*}
\norm{u}_{L^2_\omega B^{1/2}_{\Phi_2,\infty} L^2_x}^2 &\lesssim \norm{u}_{L^2_\omega L^{\Phi_2}_t L^2_x}^2 + \mathbb{E} \left[ \seminorm{u}_{B^{1/2}_{\Phi_2,\infty}(I;L^2_x)}^2 \right] \\
&\lesssim \norm{u}_{L^{N_2}_\omega L^\infty_t L^2_x}^2 + \norm{\Pi_{\Div} \Div S(\varepsilon u)}_{L^2_\omega L^2_t L^2_x}^2 + 1.
\end{align*}
We have verified~\eqref{eq:mainNikolski}.

Part~\ref{it:a} follows similarly to part~\ref{it:b}, but we use the embedding $L^\infty_t \hookrightarrow L^q_t$ and Theorem~\ref{app:StochInt thm:Stability}~\ref{app:StochInt it:2} together with an extrapolation result, cf.~\cite[Remark~3.4]{MR4116708}.
\end{proof}

\begin{remark}
Theorem~\ref{ch:Stokes thm:BesovReg} immediately implies H\"older regularity up to almost $1/2$ by an application of a Sobolev embedding (cf.~\cite[Section~2.3.2]{MR781540}), i.e., for any $\alpha \in (0,1/2)$
\begin{align*}
u \in L^2_{\omega} B^{1/2}_{\Phi_2,\infty} L^2_x \hookrightarrow L^2_{\omega} C^{0,\alpha}_t L^2_x.
\end{align*}
\end{remark}

\subsection{Nikolskii-regularity of non-linear gradient}
Before we can turn our attention to investigate the regularity properties of the non-linear gradient, we need to understand how the stochastic integral operator interacts with linear projections and gradients. 

\begin{lemma} \label{ch:Stokes lem:GradientStabilityStoch}
Let $p \geq 2$ and Assumption~\ref{ch:Stokes ass:NoiseCoeffstrong} be satisfied. Then there exists a constant $C > 0$ such that for all progressively measurable $u \in L^{p}_\omega L^\infty_t \big( L^2_{\Div} \cap W^{1,p}_{0,x}\big)$ it holds
\begin{align} \label{ch:Stokes eq:GradientStabilityStoch}
\norm{\varepsilon \Pi_{\Div} \mathcal{I}\big(G(u)\big)}_{L^p_\omega B^{1/2}_{p,\infty} L^p_x}^p \leq C\big( 1 + \norm{\varepsilon u}_{L^{p}_\omega L^\infty_t L^p_x}^p \big).
\end{align}
\end{lemma}

\begin{proof}
Due to the linearity of $\varepsilon \Pi_{\Div}$ and $\mathcal{I}$ we find
\begin{align*}
\varepsilon \Pi_{\Div} \mathcal{I}\big(G(u)\big) =  \mathcal{I}\big(\varepsilon \Pi_{\Div}G(u)\big).
\end{align*}
Therefore, using the time regularity of stochastic integrals, cf. Theorem~\ref{app:StochInt thm:Stability}~\ref{app:StochInt it:2},
\begin{align*}
\norm{\varepsilon \Pi_{\Div} \mathcal{I}(G(u))}_{L^p_\omega B^{1/2}_{p,\infty} L^p_x}^p &= \norm{ \mathcal{I}( \varepsilon \Pi_{\Div} G(u))}_{L^p_\omega B^{1/2}_{p,\infty} L^p_x}^p \\
&\lesssim \norm{ \varepsilon \Pi_{\Div} G(u)}_{L^p_\omega L^\infty_t \gamma(U; L^p_x)}^p.
\end{align*}
Since $L^p_x$ is of $2$-smooth, we can use~\eqref{app:StochInt eq:RelGammaL2}. This and the growth assumption~\eqref{ch:Stokes ass:growthstrong} establish
\begin{align*}
\norm{  \varepsilon \Pi_{\Div} G(u)}_{L^{p}_\omega L^{\infty}_t \gamma(U;L^p_x)}^p \lesssim \norm{  \varepsilon \Pi_{\Div} G(u)}_{L^{p}_\omega L^{\infty}_t L_2(U;L^p_x)}^p \lesssim \norm{\varepsilon u}_{L^{N_p}_\omega L^{\infty}_t L^p_x}^p + 1.
\end{align*}
\end{proof}

Now we are ready to prove the result on temporal regularity of the non-linear gradient.

\begin{proof}[Proof of Theorem~\ref{ch:Stokes thm:NikolskiiGrad}]
The $V$-coercivity, cf. Lemma~\ref{app:FuncSpace lem:hammer}, allows to rewrite
\begin{align} \label{ch:Stokes eq:VNon}
\norm{\tau_h V(\varepsilon u)}_{L^2_x}^2 &\eqsim \left( \tau_h S(\varepsilon u), \tau_h \varepsilon u \right).
\end{align}
It follows, using the symmetry and integration by parts,
\begin{align*}
\left( \tau_h S(\varepsilon u), \tau_h \varepsilon u \right) = \left( \tau_h S(\varepsilon u), \tau_h \nabla u \right) = -\left( \tau_h \Pi_{\Div} \Div S(\varepsilon u), \tau_h u \right).
\end{align*}
Applying the strong formulation~\eqref{ch:Stokes eq:StrongFormulation}
\begin{align*}
&\left( \tau_h \Pi_{\Div} \Div S(\varepsilon u), \tau_h u \right) \\
&\hspace{2em }= \left( \tau_h \Pi_{\Div} \Div S(\varepsilon u), \int_{\cdot - h}^\cdot \Pi_{\Div} \Div S(\varepsilon u) \dd \tau \right) \\
&\hspace{4em} + \left( \tau_h \Pi_{\Div} \Div S(\varepsilon u), \tau_h \Pi_{\Div} \mathcal{I}(G(u)) \right)\\
&\hspace{2em } =: \mathrm{I} + \mathrm{II}.
\end{align*}

The first term is estimated by H\"older's and Young's inequalities
\begin{align} \label{ch:Stokes eq:FirstSplit}
\begin{aligned}
\mathrm{I} &\leq \norm{\tau_h \Pi_{\Div} \Div S(\varepsilon u)}_{L^2_x} \norm{\int_{\cdot - h}^\cdot \Pi_{\Div} \Div S(\varepsilon u) \dd \tau } \\
&\leq h \left( \norm{\tau_h \Pi_{\Div} \Div S(\varepsilon u)}_{L^2_x}^2  + \dashint_{\cdot-h}^{\cdot} \norm{\Pi_{\Div} \Div S(\varepsilon u) }_{L^2_x}^2 \dd \tau \right).
\end{aligned}
\end{align}

The stochastic term needs a more refined analysis. Due to the symmetry
\begin{align*}
\mathrm{II} &= -\left( \tau_h S(\varepsilon u), \tau_h \nabla \Pi_{\Div} \mathcal{I}(G(u)) \right) = - \left( \tau_h S(\varepsilon u), \tau_h \varepsilon \Pi_{\Div} \mathcal{I}(G(u)) \right).
\end{align*}
Using Lemma~\ref{app:FuncSpace lem:hammer} and Young's inequality
\begin{align*}
\mathrm{II} &\leq \int_{\mathcal{O}} \abs{\tau_h S(\varepsilon u)} \abs{\tau_h \varepsilon \Pi_{\Div} \mathcal{I}(G(u))} \dd x \\
&\eqsim \int_{\mathcal{O}} \left( \kappa + \abs{\varepsilon u} + \abs{\tau_h \varepsilon u} \right)^{p-2} \abs{\tau_h \varepsilon u} \abs{\tau_h \varepsilon \Pi_{\Div} \mathcal{I}(G(u))} \dd x \\
&\leq \delta \norm{\tau_h V(\varepsilon u)}_{L^2_x}^2 + c_\delta \int_{\mathcal{O}} \left( \kappa + \abs{\varepsilon u} + \abs{\tau_h \varepsilon u} \right)^{p-2} \abs{\tau_h \varepsilon \Pi_{\Div} \mathcal{I} (G(u))}^2 \dd x\\
&=: \mathrm{II}_1 + \mathrm{II}_2.
\end{align*}
The first term can be absorbed. The remaining term is estimated by H\"older's and Young's inequalities
\begin{align} \label{ch:Stokes eq:FirstSplit02}
\begin{aligned}
\mathrm{II}_2 &\leq h c_\delta\left( \frac{p-2}{p}  \int_{\mathcal{O}} (\kappa + \abs{\varepsilon u} + \abs{\tau_h \varepsilon u})^p \dx + \frac{2}{p} h^{-p/2} \norm{\tau_h \varepsilon \Pi_{\Div} \mathcal{I} (G(u))}_{L^p_x}^p  \right) \\
&\lesssim h c_\delta\left( \frac{p-2}{p} \left( \sup_{t \in I} \norm{\varepsilon u}_{L^p_x}^p + 1 \right) + \frac{2}{p} h^{-p/2} \norm{\tau_h \varepsilon \Pi_{\Div} \mathcal{I} (G(u))}_{L^p_x}^p  \right).
\end{aligned}
\end{align}

We are in position to estimate the non-linear expression $V(\varepsilon u)$. Integrate~\eqref{ch:Stokes eq:VNon} in time and rescale by $h^{-1}$. Moreover, take the supremum over $h$ and expectation. Finally, apply~\eqref{ch:Stokes eq:FirstSplit} and~\eqref{ch:Stokes eq:FirstSplit02} to find
\begin{align*}
&(1-\delta)\mathbb{E} \left[ \sup_{h \in I} h^{-1/2} \int_{I_h} \norm{\tau_h V(\varepsilon u)}_{L^2_x}^2 \dd s \right] \\
&\hspace{2em }\lesssim \mathbb{E} \left[ \sup_{h \in I} \int_{I_h} \left\{ \norm{\tau_h \Pi_{\Div} \Div S(\varepsilon u)}_{L^2_x}^2  + \dashint_{\cdot-h}^{\cdot} \norm{\Pi_{\Div} \Div S(\varepsilon u) }_{L^2_x}^2 \dd \tau \right\} \dd s\right] \\
&\hspace{2em} + c_\delta \left( \mathbb{E} \left[ \sup_{t\in I} \norm{\varepsilon u}_{L^p_x}^p \right] + 1 + \mathbb{E} \left[\sup_{h\in I} h^{-p/2} \int_{I_h} \norm{\tau_h \varepsilon \Pi_{\Div} \mathcal{I}\big( G(u)\big)}_{L^p_x}^p \dd s \right] \right).
\end{align*}
Fubini's theorem and the a priori estimate~\eqref{ch:Stokes eq:StrongEnergy} imply
\begin{align*}
&\mathbb{E} \left[ \sup_{h \in I} \int_{I_h} \left\{ \norm{\tau_h \Pi_{\Div} \Div S(\varepsilon u)}_{L^2_x}^2  + \dashint_{\cdot-h}^{\cdot} \norm{\Pi_{\Div} \Div S(\varepsilon u) }_{L^2_x}^2 \dd \tau \right\} \dd s\right] \\
&\hspace{2em} \lesssim \mathbb{E} \left[ \int_{I} \norm{ \Pi_{\Div} \Div S(\varepsilon u)}_{L^2_x}^2 \dd s\right].
\end{align*}

The time regularity of the stochastic integral, cf. Lemma~\ref{ch:Stokes lem:GradientStabilityStoch}, allows to conclude
\begin{align*}
&\mathbb{E} \left[\sup_{h\in I} h^{-p/2} \int_{I_h} \norm{\tau_h \varepsilon \Pi_{\Div} \mathcal{I}\big( G(u)\big)}_{L^p_x}^p \dd s \right] \\
&\hspace{2em} = \mathbb{E} \left[ \seminorm{\varepsilon \Pi_{\Div} \mathcal{I}\big( G(u) \big)}_{B^{1/2}_{2,\infty} L^p_x}^p \right] \lesssim  \norm{\varepsilon u}_{L^{p}_\omega L^\infty_t L^p_x}^p + 1.
\end{align*}
Choosing $\delta > 0$ sufficiently small and~\eqref{eq:EquivalenceShift} lead to
\begin{align*}
\mathbb{E}\left[ \seminorm{V(\varepsilon u)}_{ B^{1/2}_{2,\infty} L^2_x}^2 \right] &\lesssim \mathbb{E} \left[ \int_{I} \norm{ \Pi_{\Div} \Div S(\varepsilon u)}_{L^2_x}^2 \dd s\right]+  \norm{\varepsilon u}_{L^{p}_\omega L^\infty_t L^p_x}^p + 1 \\
&\lesssim  \mathbb{E} \left[ \int_{I} \norm{ \Pi_{\Div} \Div S(\varepsilon u)}_{L^2_x}^2 \dd s\right]+ \mathbb{E} \left[ \sup_{t\in I} \mathcal{J}(u) \right]+ 1.
\end{align*}

Lastly, using the equivalence $\varphi'(t) t \eqsim \varphi(t)$, 
\begin{align*}
\mathbb{E} \left[ \norm{V(\varepsilon u)}_{L^2_t L^2_x}^2 \right] &= \mathbb{E} \left[ \int_I \int_\mathcal{O} \varphi_{\kappa}'(\abs{\varepsilon u}) \abs{\varepsilon u} \dd x \dd s \right] \\
&\eqsim \mathbb{E} \left[ \int_I \int_\mathcal{O} \varphi_{\kappa}(\abs{\varepsilon u}) \dd x \dd s \right]  = \mathbb{E} \left[ \int_I \mathcal{J}(u) \dd s \right].
\end{align*}
All together, we proved
\begin{align*}
\norm{V(\varepsilon u)}_{L^2_\omega B^{1/2}_{2,\infty} L^2_x}^2 \lesssim  \mathbb{E} \left[ \int_{I} \norm{ \Pi_{\Div} \Div S(\varepsilon u)}_{L^2_x}^2 \dd s\right]+ \mathbb{E} \left[ \sup_{t\in I} \mathcal{J}(u) \right]+ 1.
\end{align*}

\end{proof}

\begin{remark}
Strong solutions to the symmetric stochastic $p$-Stokes system~\eqref{ch:Stokes intro:p-Stokes-stoch} are comparable in terms of temporal regularity to the stochastic $p$-Laplace system, cf.~\cite{wichmann2021temporal}. 

In particular, if one could prove increased integrability, e.g. $ \Pi_{\Div}\Div S(\varepsilon u) \in L^2_{\omega} L^{q}_t L^2_x$ for some $q > 2$, then it would also be possible to show $V(\varepsilon u) \in L^2_\omega B^{1/2}_{q,\infty} L^2_x$.

The temporal regularity enables a similar error analysis for time discretizations of~\eqref{ch:Stokes intro:p-Stokes-stoch} as done in e.g.~\cite{MR4298537,Diening2022} for the $p$-Laplace system.
\end{remark}

\appendix
\section{Fourier representation of Besov norms} \label{app:Besov}
There are many different characterizations of Besov spaces. A nice overview can be found e.g. in the books of Triebel~\cite{MR781540,MR1163193}. Some results derived from a specific characterization are not obvious for others. In this appendix we introduce the Fourier representation of Besov spaces and state some properties.

\begin{definition} \label{def:partitionOfUnity}
A family of functions $\set{\varphi_j}_{j\in \mathbb{N}_0}$ is called smooth dyadic partition of unity, if the following conditions are satisfied:
\begin{enumerate}
\item $\varphi_j \in C^\infty(\mathbb{R})$, 
\item \label{it:support}
\begin{align*}
\begin{cases}
\mathrm{supp}\, \varphi_0 \subset \{t \in \mathbb{R}: \,\abs{t} \leq 2 \}, \\
\mathrm{supp}\, \varphi_j \subset \{t \in \mathbb{R}: \,2^{j-1} \leq \abs{t} \leq 2^{j+1} \}, \quad j \in \mathbb{N},
\end{cases}
\end{align*}
\item \label{it:bounds} for any $k \in \mathbb{N}_0$
\begin{align*}
\sup_{t\in \mathbb{R}, j\in \mathbb{N}_0} 2^{jk} \abs{\partial^k \varphi_j(t)} < \infty,
\end{align*}
\item \label{it:ones} and for any $t \in \mathbb{R}$
\begin{align*}
\sum_{j\in \mathbb{N}_0} \varphi_j(t) = 1.
\end{align*}
\end{enumerate}
\end{definition}

Now we are in a good shape to define the Fourier representation of Besov spaces on the full space~$\mathbb{R}$ with values in a Banach space~$X$.
\begin{definition}
Let $\alpha \in \mathbb{R}$, $p \in [1,\infty]$ and  $q\in[1,\infty]$. We define
\begin{align} \label{def:Besov02}
\norm{f}_{\tilde{B}^{\alpha}_{p,q}(\mathbb{R};X)} := \norm{\set{2^{j\alpha} \norm{\mathcal{F}^{-1}\big(\varphi_j \mathcal{F}(f) \big) }_{L^{p}(\mathbb{R} ;X)} } }_{\ell_q}
\end{align}
and
\begin{align} \label{def:spaceBesov02}
\tilde{B}^{\alpha}_{p,q}(\mathbb{R};X) := \left\{ u: I \to X |\, u \text{ Bochner-measurable,}\, \norm{u}_{\tilde{B}^{\alpha}_{p,q}(\mathbb{R};X)} < \infty \right\}.
\end{align}
Here $\mathcal{F}$ denotes the Fourier transform.
\end{definition}

In contrast to the definition~\eqref{def:Besov01}, where integrability of scaled increments is measured,~\eqref{def:Besov02} measures integrability of frequency localized and scaled values. Note that~\eqref{def:Besov02} is defined explicitly even for $\alpha <0$. A generalization of the explicit definition~\eqref{def:Besov01} for negative $\alpha$ is not clear. Nevertheless, for $\alpha > 0$ one can verify equivalence of~\eqref{def:Besov01} and~\eqref{def:Besov02} as done in~\cite[Section~1.5.1]{MR1163193}.

\begin{theorem} \label{thm:spectralBesov}
Let $\alpha > 0$, $p \in (1,\infty)$ and $q \in [1,\infty]$. Then $ B^{\alpha}_{p,q}(\mathbb{R};X) =  \tilde{B}^{\alpha}_{p,q}(\mathbb{R};X)$ and for all $f \in B^{\alpha}_{p,q}(\mathbb{R};X)$ it holds
\begin{align}
\norm{f}_{\tilde{B}^{\alpha}_{p,q}(\mathbb{R};X)}  \eqsim \norm{f}_{B^{\alpha}_{p,q}(\mathbb{R};X)}.
\end{align}
\end{theorem}
\begin{corollary}\label{cor:spectralBesov}
Let $\alpha > 0$, $p \in (1,\infty)$ and $q \in [1,\infty]$. Then $ \big( B^{\alpha}_{p,q}(\mathbb{R};X) \big)' = \big( \tilde{B}^{\alpha}_{p,q}(\mathbb{R};X)\big)'$ with equivalent norms.
\end{corollary}

The Fourier approach allows to characterize the dual spaces, cf.~\cite[Section~2.11.2]{MR781540}. 

\begin{lemma} \label{lem:dual}
Let $X$ be reflexive, $\alpha \in \mathbb{R}$, $p\in [1,\infty)$ and $q \in [1,\infty)$. Then $\big(\tilde{B}^{\alpha}_{p, q}(\mathbb{R};X) \big)' = \tilde{B}^{-\alpha}_{p',q'}(\mathbb{R};X')$.
\end{lemma}

Additionally, we can control the interaction of derivatives and the differentiability parameter in the Besov-norm.
\begin{lemma} \label{app:lem_der}
Let $\alpha \in \mathbb{R}$, $p \in [1,\infty]$ and $q \in [1,\infty]$. Additionally, assume $u \in \tilde{B}^{\alpha+1}_{p,q}(\mathbb{R};X)$. Then $\partial_t u \in \tilde{B}^{\alpha}_{p,q}(\mathbb{R};X)$ and
\begin{align} \label{eq:DerivativesAndNorms}
\norm{\partial_t u}_{\tilde{B}^{\alpha}_{p,q}(\mathbb{R};X)} \lesssim \norm{u}_{\tilde{B}^{\alpha + 1}_{p,q}(\mathbb{R};X)}.
\end{align}
\end{lemma}
\begin{proof}
Notice by~\cite[Theorem in Section~2.3.8]{MR781540} we have the equivalence
\begin{align*}
\norm{u}_{\tilde{B}^{\alpha + 1}_{p,q}(\mathbb{R};X)} \eqsim \norm{u}_{\tilde{B}^{\alpha }_{p,q}(\mathbb{R};X)} + \norm{\partial_t u}_{\tilde{B}^{\alpha }_{p,q}(\mathbb{R};X)}.
\end{align*}
The claim follows.
\end{proof}

%% New biber style  (Remove \i!!!)
\printbibliography 

%\bibliographystyle{amsalpha}
%\bibliography{numerics}

\end{document}